\documentclass[twoside,12pt]{amsart}

\usepackage{amsmath,latexsym,amssymb,mathptm, times,verbatim, mathcomp}
    \usepackage{longtable}\usepackage{tikz,stmaryrd}

\usepackage{enumerate,amssymb}

\input amssym.def
\input amssym
\DeclareMathSymbol{\subsetneq}{\mathrel}{AMSb}{"28}
\input xypic
\input xy
\xyoption{all}
\def\Cl{\operatorname{Cl}}
\def\htt{\operatorname{ht}}

\def\depth{\operatorname{depth}}

\def\depth{\operatorname{depth}}
\def\Spec{\operatorname{Spec}}
\def\Ass{\operatorname{Ass}}

\def\Ext{\operatorname{Ext}}
\def\Tor{\operatorname{Tor}}
\def\im{\operatorname{im}}
\def\Pf{\operatorname{Pf}}

\def\rank{\operatorname{rank}}

\def\coker{\operatorname{coker}}

\define\Hom{\operatorname{Hom}}

\def\Sym{\operatorname{Sym}}

\def\ann{\operatorname{ann}}

\def\grade{\operatorname{grade}}
\def\pd{\operatorname{pd}}

\def\HH{\operatorname{H}}

\newtheorem{theorem}{Theorem}[section]
\newtheorem{lemma}[theorem]{Lemma}
\newtheorem{corollary}[theorem]{Corollary}
\newtheorem{proposition}[theorem]{Proposition}

\newtheorem{dream lemma}[theorem]{Dream Lemma}

\newtheorem{observation}[theorem]{Observation}

\theoremstyle{definition}
\newtheorem{definition}[theorem]{Definition}
\newtheorem{remark}[theorem]{Remark}
\newtheorem{remark-no-advance}[equation]{Remark}
\newtheorem{remarks-no-advance}[equation]{Remarks}
\newtheorem{claim-no-advance}[equation]{Claim}
\newtheorem{proposition-no-advance}[equation]{Proposition}
\newtheorem{corollary-no-advance}[equation]{Corollary}

\newtheorem{remarks}[theorem]{Remarks}
\newtheorem{data}[theorem]{Data}

\newtheorem{notation}[theorem]{Notation}

\newtheorem{present summary}[theorem]{Present Summary}

\newtheorem{example}[theorem]{Example}

\newtheorem{chunk}[theorem]{}

\newtheorem{marching orders}[theorem]{Marching Orders}

\newtheorem{circle the wagons}[theorem]{Circle the wagons}

\newtheorem*{proof-of-claim-1}{Proof of Claim~\ref{claim1}}
\newtheorem*{proof-of-claim-2}{Proof of Claim~\ref{claim2}}
\newtheorem*{proof-of-key}{Proof of Lemma~\ref{key}}
\newtheorem*{proof-of-claim-3}{\rm Proof of Claim~\ref{claim3}}

\numberwithin{equation}{theorem}
\numberwithin{table}{theorem}
\numberwithin{figure}{theorem}

\begin{document}

\baselineskip=16pt

\title[An alternating matrix and a vector]{ 
An  alternating matrix and a vector,\\ with application to Aluffi algebras}

\date\today 
\author[Andrew R. Kustin]
{Andrew R. Kustin}

\thanks{AMS 2010 {\em Mathematics Subject Classification}.
Primary 13H10; Secondary 13A30}

\thanks{The 
author was partially supported by the the Simons Foundation
grant number 233597.}


\address{Department of Mathematics, University of South Carolina,
Columbia, SC 29208} \email{kustin@math.sc.edu}

\begin{abstract} Let $\mathbf X$ be a generic alternating matrix, $\mathbf t$ be a generic row vector, and $J$ be the ideal $\Pf_4({\mathbf X})+I_1({\mathbf {t X}})$. We prove that $J$ is a perfect Gorenstein ideal of grade equal to the grade of $\Pf_4({\mathbf X})$ plus two. This result is used by Ramos and Simis in their calculation of the Aluffi algebra of the module of derivations of the homogeneous coordinate ring of a smooth projective hypersurface. 
We also prove that $J$ defines a domain, or a normal ring, or a unique factorization domain if and only if the base ring has the same property. The main object of study in the present paper is the module $\mathcal N$ which is equal to the column space of 
$\mathbf X$, calculated mod $\Pf_4({\mathbf X})$. The module $\mathcal N$ is a self-dual maximal Cohen-Macaulay module of rank two; furthermore, $J$ is a Bourbaki ideal for  $\mathcal N$. The ideals which define 
the homogeneous coordinate rings of the Pl\"ucker embeddings of the  
Schubert subvarieties of the Grassmannian of planes are used in the study of the module $\mathcal N$. 
\end{abstract}

\maketitle

\tableofcontents

\bigskip

\section{Introduction.} 

\bigskip

Aluffi \cite{A04}  introduced a  class of algebras which are intermediate between
the symmetric algebra and the Rees algebra of an ideal in order to define the characteristic
cycle of a hypersurface parallel to the conormal cycle in intersection theory.
These algebras were  investigated by  Nejad and  Simis \cite{NS11}, who called them Aluffi algebras. At the end of his paper, Aluffi observed that it would be computationally desirable to up-grade his methods to more general schemes. A first step in the direction of Aluffi's proposed up-grade is a good notion of the Rees algebra of a module, such as the one described by Simis, Ulrich, and Vasconcelos in \cite{SUV03}. A second step in the direction of Aluffi's proposed up-grade is a good notion of the Aluffi algebra of a module, such as the one the one introduced by Ramos and Simis in \cite{RS15}.

Ramos and Simis compute the Aluffi algebra of the module of derivations of the homogeneous coordinate ring of a smooth projective hypersurface. In other contexts this module is also called the module of tangent vector fields or the differential idealizer or the module of logarithmic
derivations. As part of the Ramos-Simis program, it is necessary to understand the homological nature of the ideal $J=\Pf_4({\mathbf X})+I_1({\mathbf t}{\mathbf X})$, where ${\mathbf X}$ is a generic alternating matrix and ${\mathbf t}$ is a generic row vector. Simis told us that he and Ramos conjectured that $J$ is a Gorenstein ideal of height two more than the height of $\Pf_4({\mathbf X})$. The purpose of this paper is to prove the Ramos-Simis conjecture.

Let $R_0$ be an arbitrary commutative Noetherian ring, $\goth f$ be an integer with $4\le \goth f$,  $$\mathcal R
=R_0[\{x_{i,j}\mid 1\le i<j\le \goth f\}\cup \{t_i\mid 1\le i\le \goth f\}]$$ be 
a polynomial ring in $\binom{\goth f}2+\goth f$ indeterminates, $\mathbf X$ be the be the $\goth f\times \goth f$ alternating matrix with with $x_{i,j}$ in position (row $i$, column $j$) for $i<j$, $\mathbf t$ be the $1\times \goth f$ matrix with $t_j$ in column $j$, $I$ be the ideal $\Pf_4(\mathbf{X})$ which is generated   
by the set of Pfaffians of the principal $4\times 4$ submatrices of $\mathbf X$,
$K$ be the ideal $
I_1(\mathbf {tX})$, which is 
generated by the entries of the product of $\mathbf t$ times $\mathbf X$, and  
$J$ be the ideal $I+K$ of $\mathcal R$.
The main result in the paper, Theorem~\ref{main-Theorem}, is that $J$ is a perfect Gorenstein ideal in $\mathcal R$ of grade $\binom{\goth f-2}2+2$. In particular, if $R_0$ is a Gorenstein ring, then $\mathcal R/J$ is a Gorenstein ring. Some consequences of the main result are contained in Corollary~\ref{corollary} where it is shown that $\mathcal R/J$ is a domain, or a normal ring, or a unique factorization domain if and only if the base ring $R_0$ has the same property.

The main ingredient  in the proof of Theorem~\ref{main-Theorem} takes place over the polynomial ring $$R
=R_0[\{x_{i,j}\mid 1\le i<j\le \goth f\}].$$ We prove in Lemma~\ref{main-Dream-Lemma}  that ``the column space of $\mathbf X$, calculated mod $I$'', which is equal to the submodule
\begin{equation}\label{column}\textstyle \{{\mathbf X}\theta\in (\frac RI)^\goth f\mid \theta\in (\frac RI)^\goth f\} \quad\text{of $(\frac RI)^\goth f$,}\end{equation}
is a perfect $R$-module of projective dimension $\binom{\goth f-2}2$.
If $R_0$ is a Cohen-Macaulay domain, then we prove, in Observation~\ref{yet-to-come}.\ref{yet.d}, that
the module of (\ref{column}) is a self-dual maximal Cohen-Macaulay $R/I$-module of rank two; 
and we prove, see Remark~\ref{R4}.\ref{here it is}, 
that the ideal $J(\mathcal R/I)$, which is the central object in this paper, is a 
 Bourbaki ideal for $\mathcal R\otimes_R\text{(\ref{column})}$; and therefore, homological properties of (\ref{column}) are
inherited by  $\mathcal R/J$. For more discussion about Bourbaki ideals, see, for example, \cite{A66, M80, BHU87, SUV03}. 

The ideal $I$ is the ideal of ``quadratic relations'' which define the homogeneous coordinate ring of the image of the Pl\"ucker embedding of the Grassmannian  $\operatorname{Gr}(2,\goth f)$ into projective space $\mathbb P^{\binom{\goth f} 2-1}$. Properties of the Schubert subvarieties of $\operatorname{Gr}(2,\goth f)$ play a crucial role in our proof of the properties of the module 
 (\ref{column}). 

The ideal $I_1(\mathbf{tX})$ has already been studied. If $\goth f$ is odd, then 
$I_1(\mathbf{tX})$  is a type two almost complete intersection ideal introduced by Huneke and Ulrich in \cite{HU85} and further studied in \cite{K95}. If $\goth f$ is even, then $I_1(\mathbf{tX})$  is a mixed ideal; its unmixed part is $I_1(\mathbf{tX})+\Pf_{\goth f}({\mathbf X})$; see, for example, \cite{KPU-Ann}. This ideal is a deviation two, grade $\goth f-1$ Gorenstein ideal also introduced in \cite{HU85} and further studied in \cite{K86,S91,K92}.

\bigskip

\section{Notation, conventions, and preliminary results.}\label{Prelims}

\bigskip
\begin{chunk}Let $M$ and $N$ be  modules over a commutative Noetherian ring $R$. Whenever the meaning is unambiguous, we write $M^*$, $M\otimes N$, $\Hom(M,N)$, and $\bigwedge^i M$ in place of $\Hom_R(M,R)$, ${M\otimes_R N}$, $\Hom_R(M,N)$, and $\bigwedge_R^i M$, respectively.
\end{chunk} 

\begin{chunk} An element $x$ of a ring $R$ is {\em regular} on the $R$-module $M$ if $x$ is a non-zero-divisor on $M$. In other words, if $xm=0$ for some element $m\in M$, then $m=0$. \end{chunk} 

\begin{chunk}If $x$ is a non-nilpotent element of a commutative Noetherian ring $R$, then {\em the localization of $R$ at $x$}, denoted $R_x$, is the ring $S^{-1}R$ where $S$ is the set $\{1,x,x^2,x^3,\dots\}$. If $x$ is a regular element of $R$, then we use the notation $R_x$ and $R[x^{-1}]$ interchangeably. 
\end{chunk} 

\begin{chunk}We denote the ring of integers by $\mathbb Z$.
\end{chunk} 

\subsection{Perfection}\label{Terminology}

\bigskip
\begin{chunk}\label{perfection} Let $R$ be a Noetherian ring,  $I$ be a proper ideal of $R$, and $M$ be a non-zero finitely generated $R$-module. \begin{enumerate}[\rm(a)]  
\item The {\it grade} of $I$ is the length of a maximal regular sequence on $R$ which is 
contained in $I$. (If $R$ is Cohen-Macaulay, then the grade
of $I$ is equal to the height of $I$.) 
\item The $R$-module $M$ is called {\it perfect} if the grade of the annihilator of $M$ (denoted $\operatorname{ann}_R M$) is equal to the projective dimension of $M$ (denoted $\operatorname{pd}_R M$). The inequality \begin{equation}\label{auto}\grade(\operatorname{ann}_R M)\le \operatorname{pd}_R M\end{equation} holds automatically if $M\not=0$.
\item\label{constant-pd} If $M$ is a perfect $R$-module, then
$$\pd_{R_P}M_P=\grade \ann_{R_P}M_P=\grade\ann_RM$$ for all prime ideals $P$ in the support of $M$. (See, for example, \cite[Prop.~16.17]{BV}.)
\item \label{perfection.d} If $R$ is a polynomial ring over a field or over the ring of integers and $M$ is a finitely generated graded $R$-module, then  
$M$ is a perfect $R$-module if and only if $M$ is a Cohen-Macaulay $R$-module. (This is not the full story. For more information, see, for example, \cite[Prop.~16.19]{BV} or \cite[Thm.~2.1.5]{BH}.)

\item\label{gor} The ideal $I$  in  $R$ is called a {\it perfect ideal} if $R/I$ is a perfect $R$-module.
 A perfect ideal $I$ of grade $g$ is a {\it Gorenstein ideal}
 if $\operatorname{Ext}^g_R(R/I,R)$ 
is a cyclic R-module. 
\end{enumerate}
\end{chunk}

The concept of perfection is particularly useful because of the   ``Persistence of Perfection Principle'', which is also known as  the ``transfer of perfection''; see \cite[Prop. 6.14]{Ho75} or \cite[Thm. 3.5]{BV}.

\begin{theorem}\label{PoPP}
Let $R\to S$ be a homomorphism of Noetherian rings, $M$ be a perfect   $R$-module,  and   $\mathbb P$ be a resolution of $M$ by projective $R$-modules. If $S\otimes_RM\neq 0$ and $$\grade (\ann M)\le \grade (\ann(S\otimes_RM)),$$ 
then $S\otimes_RM$ is a perfect $S$-module with $\pd_S(S\otimes_RM)=\pd_RM$ and  $S\otimes_R\mathbb P$ is a  resolution of $S\otimes_RM$ by projective $S$-modules.
\end{theorem}

\subsection{Multilinear algebra.}

 \begin{chunk} \label{87Not1} 
Many of our calculations are made in a coordinate-free manner. If the calculation is coordinate free, then the signs take care of themselves. 
In particular, when working with Pfaffians, we prefer to use elements of an exterior algebra rather than to define and keep track of sign conventions which mimic operations that take place in an exterior algebra.

Let $R$ be a commutative Noetherian ring and $F$ be a free module of finite rank $\goth f$ over 
$R$.  
We make much use of 
the exterior algebras $\bigwedge^{\bullet} F$ and $\bigwedge^{\bullet} F^*$,  
the fact that 
$\bigwedge^{\bullet} F$ and $\bigwedge^{\bullet} F^*$ are modules over one another, and the fact that the even part of an exterior algebra comes equipped with a divided power structure.  
The
rules for a divided power algebra are recorded in \cite[section 7]{GL}  or  \cite[Appendix 2]{Ei95}.
(In practice these rules say that $w^{
(a)}$ behaves like $w^
a/(a!)$ would behave if $a!$ were
a unit in $R$.) 
\end{chunk}

\begin{chunk}\label{2.3} We recall some of the properties of the divided power structure on the subalgebra $\bigwedge^{2\bullet}F$ of the exterior algebra $\bigwedge^{\bullet}F$. Suppose that $e_1,\dots,e_{\goth f}$ is a basis for the free $R$-module $F$ and $$f_2=\sum_{1\le i_1<i_2\le \goth f}a_{i_1,i_2} \ e_{i_1}\wedge e_{i_2}$$ is an element of $\bigwedge^2F$, for some $a_{i_1,i_2}$ in $R$. 
Let $A$ be the $\goth f\times \goth f$ alternating matrix with
$$A_{i,j}=\begin{cases} a_{i,j},&\text{if $i<j$,}\\0,&\text{if $i=j$, and}\\-a_{i,j},&\text{if $j<i$}.\end{cases}$$
For each positive integer $\ell$, the $\ell$-th divided power of $f_2$ is 
$$f_2^{(\ell)}=\sum\limits_I A_I e_I\in\textstyle\bigwedge^{2\ell}F,$$ where the $2\ell$-tuple $I=(i_1,\dots,i_{2\ell})$ roams over all increasing sequences  of integers with $1\le i_1$ and $i_{2\ell}\le \goth f$, $e_I=e_{i_1}\wedge \dots \wedge e_{i_{2\ell}}$, and $A_I$ is the Pfaffian of the submatrix of $A$ which consists of rows and columns $\{i_1,\dots,i_{2\ell}\}$, in the given order.  Furthermore, $\bigwedge^{2\bullet}F$ is a DG$\Gamma$-module over $\bigwedge^{\bullet}F^*$. In particular, if $\tau\in F^*$ and $v_1,\dots,v_s$ are homogeneous elements of  $\bigwedge^{2\bullet}F$, then 
\begin{equation}\label{Gamma}\tau\left(v_1^{(\ell_1)}\wedge \dots\wedge v_{s}^{(\ell_s)}\right)=
\sum\limits_{j=1}^s\tau(v_j)\wedge v_1^{(\ell_1)}\wedge \dots\wedge
v_j^{(\ell_j-1)}
\wedge \dots\wedge v_{s}^{(\ell_s)}.\end{equation} For more details see, for example, 
\cite[Appendix~A2.4]{Ei95} or \cite[Appendix and Sect. 2]{BE77}.
\end{chunk}

The following   fact about the interaction of the module structures of 
$\bigwedge^\bullet F $ on  $\bigwedge^\bullet F^* $ and $\bigwedge^\bullet F^* $ on  $\bigwedge^\bullet F $   
 is well known; see \cite[section 1]{BE75} and \cite[Appendix]{BE77}.

\begin{proposition} \label{A3} Let $F$ be a free module of finite rank  
over a commutative
Noetherian ring $R$. If $f_1\in 
F$, $f_p\in \textstyle \bigwedge^{p}F$, and $\phi_q\in\textstyle \bigwedge^{q}(F^{*})$, then 
$$ (f_1(\phi_q))(f_p)=f_1\wedge (\phi_q(f_p))+(-1)^{1+q}\phi_q(f_1\wedge f_p).$$ 
\end{proposition}
\vskip-24pt\qed

\medskip The following fact is important for our purposes. We prove it carefully in order to illustrate some of the ideas contained in \ref{87Not1} and \ref{2.3}.

\begin{observation} \label{doo-8.2} Let $R$ be a commutative Noetherian ring, $F$ be 
 a free $R$-module of finite rank, and $f_2$ be an element of
$\bigwedge^{2}F$.
\begin{enumerate}[\rm(a)]\item \label{doo-8.2.a} If
 $\phi_3\in \bigwedge^3F^*$, then 
$[f_2(\phi_3)](f_2)=\phi_3(f_2^{(2)})$.
\item\label{doo-8.2.b}  If $\phi_1$, $\phi_1'$, and $\phi_1''$ are in $F^*$, then 
$$f_2(\phi_1\wedge \phi_1'\wedge \phi_1'')=f_2(\phi_1\wedge \phi_1')\cdot  \phi_1''-f_2(\phi_1\wedge \phi_1'')\cdot  \phi_1'+f_2(\phi_1'\wedge \phi_1'')\cdot  \phi_1.$$\end{enumerate}
\end{observation}\begin{proof} 
We
prove (\ref{doo-8.2.a}) by showing that the two elements 
 $[f_2(\phi_3)](f_2)$ and $\phi_3(f_2^{(2)})$
of $F$ are equal 
by showing that $\phi_1\Big([f_2(\phi_3)](f_2)\Big) =\phi_1\Big(\phi_3(f_2^{(2)})\Big)$ for 
every element $\phi_1$ of $F^*$.  
Observe that
$$\begin{array}{lllllllll}\phi_1([f_2(\phi_3)](f_2))&=&-[f_2(\phi_3)][\phi_1(f_2)]
 &=&-[\phi_1(f_2)][f_2(\phi_3)]&=&-[\phi_1(f_2)\wedge f_2](\phi_3)\vspace{5pt}\\&=&-[\phi_1(f_2^{(2)})](\phi_3)&=&-\phi_3[\phi_1(f_2^{(2)})]&=&\phantom{-}\phi_1[\phi_3(f_2^{(2)})].\end{array}$$
The first and last equalities  hold because $\bigwedge^{\bullet}F$ is a module over the graded-commutative ring $\bigwedge^{\bullet}F^*$. The second and fifth equalities
follow from the fact that  the module actions of $\bigwedge^\bullet F^*$ on $ \bigwedge^\bullet F$ and $\bigwedge^\bullet F$ on $\bigwedge^\bullet F^*$ are compatible in the sense that  \begin{equation}\label{compat} \phi_i(f_i)=f_i(\phi_i)\text{ for }\phi_i\in \textstyle\bigwedge^iF^*\text{ and } f_i\in \textstyle\bigwedge^iF.\end{equation} The third equality  
is a consequence of  the module action of $\bigwedge^\bullet F$ on $\bigwedge^\bullet F^*$. The fourth equality 
is explained in 
(\ref{Gamma}). 

The proof of (\ref{doo-8.2.b}) is similar.
\end{proof}

\subsection{The set up.}

\bigskip
\begin{chunk}We set up the data in a coordinate-free manner in \ref{data2} and \ref{Not2}; a version with coordinates is given in \ref{R2}. 
The critical calculation, Lemma~\ref{main-Dream-Lemma}, involves ``$x_{i,j}$'s'', but not ``$t_i$'s''; the ambient ring for this calculation is called $R$. The information about $R$ is given in \ref{data2}.\ref{data2-one}, \ref{Not2}.\ref{Not2.a}, and \ref{R2}.\ref{R2.z}.
The main result in the paper,  Theorem~\ref{main-Theorem}, involves both ``$x_{i,j}$'s'' and ``$t_i$'s''; the ambient ring for this result is called $\mathcal R$. The ring $\mathcal R$ is an extension of $R$; the extra information about $\mathcal R$ is given in \ref{data2}.\ref{data2-both}, \ref{Not2}.\ref{Not2.b}, and \ref{R2}.\ref{R2.a} \end{chunk}

\begin{data}{\label{data2}} Let $\goth f$ be a positive integer, $R_0$  a commutative Noetherian ring,
and $V$ be a free $R_0$-module   of rank $\goth f$.
\begin{enumerate}[\rm(a)]
\item\label{data2-one} Let $R=\bigoplus_{i=0}^\infty R_i$ be the standard graded polynomial ring 
$$\textstyle R=\Sym_\bullet^{R_0}(\bigwedge_{R_0}^2V^*)$$and  $F$ be the free $R$-module $F=R\otimes_{R_0} V$. Consider the $R$-module homomorphism
$$\textstyle\xi\in \Hom_R(\bigwedge^2_R F^*, R)=\bigwedge_R^2F,$$ which is given as the composition
$$\xi:\textstyle\bigwedge^2_R F^*=R\otimes_{R_0} \bigwedge^2_{R_0} V^*=R\otimes R_1\xrightarrow{{\rm multiplication}}R.$$

\item\label{data2-both} View $\bigwedge_{R_0}^2V^*\oplus V$ as a bi-graded free $R_0$-module where each element of $\bigwedge_{R_0}^2V^*$ has degree $(1,0)$ and each element of $V$ has degree $(0,1)$. Let $\mathcal R$ be the bi-graded polynomial ring
$$\textstyle \mathcal R=\Sym_\bullet^{R_0}(\bigwedge_{R_0}^2V^*\oplus V)$$and  $\mathcal F$ be the free $\mathcal R$-module $\mathcal F=\mathcal R\otimes_{R_0} V$. Consider the $\mathcal R$-module homomorphism
$$\textstyle
\tau\in \Hom_\mathcal R(\mathcal F, \mathcal R)=\mathcal F^*$$ which is given as the composition
$$\mathcal F=\mathcal R\otimes_{R_0} V=\mathcal R\otimes \mathcal R_{\ 0,1}\xrightarrow{{\rm multiplication}}\mathcal R.$$
\item There is a natural inclusion map $\xymatrix{R\ar@{^(->}[r]&\mathcal R}$ and a natural projection map $\xymatrix{\mathcal R\ar@{->>}[r]& R}$. The $\mathcal R$-module $\mathcal F$ of (\ref{data2-both}) is also equal to $\mathcal F=\mathcal R\otimes_R F$; furthermore, the element $\xi\in\bigwedge_R^2F$ of (\ref{data2-one}) is also equal to the element $$\textstyle \xi=1\otimes \xi\text{ of }\bigwedge_{\mathcal R}^2\mathcal F=\mathcal R\otimes_R\bigwedge^2_RF.$$
\end{enumerate}\end{data}
\begin{notation}{\label{Not2}} Adopt Data  \ref{data2}. 
\begin{enumerate}[\rm(a)]\item\label{Not2.a}Let\begin{enumerate}[\rm(i)]\item $I$ be the  ideal
$
I=\im (\xi^{(2)}:\textstyle\bigwedge^4F^*\to R)$, of $R$,
\item $A$ be the ring $R/I$, \item\label{overline} $\overline{\phantom{x}}$ be the functor $A\otimes_R-$, and \item\label{Not2.a.iii} $N$ be the cokernel of the map $\overline{d_1}:\bigwedge^3\overline{F}^*\to \overline{F}^*$ where $d_1:\bigwedge^3F^*\to F^*$ is the map $d_1(\phi_3)=\xi(\phi_3)$, for $\phi_3\in \bigwedge^3F^*$.
\end{enumerate}
\item\label{Not2.b}Let \begin{enumerate}[\rm(i)]\item $K$ and $J$ be the ideals
\begin{align*}K&=\im (\tau(\xi):\mathcal F^*\to \mathcal R),\quad\text{and}\\
J&=I\mathcal R+K\end{align*} of $\mathcal R$,\item $\mathcal A$ be the ring $\mathcal R\otimes_R A$, and\item $\mathcal N$ be the $\mathcal R$-module $\mathcal R\otimes_R N$.
\end{enumerate}\end{enumerate}
\end{notation}

\begin{remark}{\label{R2}} Adopt Data~\ref{data2} and Notation~\ref{Not2}. If one picks dual bases $e_1,\dots,e_\goth f$ for $V$ and $e_1^*,\dots,e_\goth f^*$ for $V^*$, lets $x_{i,j}$ represent $e_j^*\wedge e_i^*\in \bigwedge^2_{R_0}V^*=\mathcal R_{\ 1,0}$, for $1\le i<j\le \goth f$, and lets $t_i$ represent $e_i\in V=\mathcal R_{\ 0,1}$, for $1\le i\le \goth f$, then the following statements hold.\begin{enumerate} [\rm(a)]
\item\label{R2.z} The standard graded polynomial ring $R$ is $R=R_0[\{x_{i,j}\mid 1\le i<j\le \goth f\}]$. Furthermore,
\begin{enumerate}[\rm(i)]\item\label{R2.b}  the element $\xi$ of $\bigwedge ^2F$ is $\xi=\sum\limits_{i<j} x_{i,j}\ e_i\wedge e_j$, 
\item\label{R2.e} the matrix for the $R$-module homomorphism $d_0:F^*\to F$, with $d_0(\phi_1)=\phi_1(\xi)$,  
with respect to the bases $\{e_j^*\}$ and $\{e_i\}$, is $-{\mathbf X}$, where ${\mathbf X}$ is the generic $\goth f\times \goth f$ alternating matrix 
whose    entry in position (row $i$, column $j$) is$$\begin{cases} x_{i,j}&\text{if $i<j$}\\0&\text{if $i=j$}\\-x_{j,i}&\text{if $j<i$},\end{cases}$$
and

\item\label{R2.g} the ideal $I$ of Notation \ref{Not2} is equal to $\Pf_4(\mathbf{X})$, which is the ideal of $R$ generated by the set of Pfaffians of the principal $4\times 4$ submatrices of $\mathbf X$.
\end{enumerate}
\item\label{R2.a} The bi-graded polynomial ring $\mathcal R$ is  $\mathcal
 R=R_0[\{x_{i,j}\mid 1\le i<j\le \goth f\}\cup\{t_i\mid 1\le i\le \goth f\}]$.
Furthermore, \begin{enumerate}[\rm(i)]
\item\label{R2.c}  the element  $\tau$ of $\mathcal F^*$ is  $\tau=\sum_i t_ie_i^*$ , 
\item\label{R2.d} the matrix for $\tau:\mathcal F\to \mathcal R$ with respect to the basis $\{e_i\}$ for $\mathcal F$ is the row vector $${{\mathbf t}=[t_1,\dots, t_\goth f]},$$ 
\item\label{R2.f} the element  $\tau(\xi)$ in $\mathcal F$ is an $\mathcal R$-module homomorphism $\mathcal F^*\to \mathcal R$ and the matrix for this homomorphism, with respect to the basis $\{e_i^*\}$ is the row vector ${\mathbf t}{\mathbf X}$, and 
\item\label{R2.h} the ideal $K$ of Notation \ref{Not2} is equal to $
I_1(\mathbf {tX})$, which is the ideal of $\mathcal R$ generated by the entries of the product of $\mathbf t$ times $\mathbf X$, and  
\item\label{R2.i} the ideal $J$ of Notation \ref{Not2} is equal to $\Pf_4(\mathbf {X})\cdot \mathcal R+
I_1(\mathbf {tX})$.
\end{enumerate} \end{enumerate}\end{remark}

\begin{remark}\label{promise}Adopt the language of \ref{data2}.\ref{data2-one} and \ref{Not2}.\ref{Not2.a}.  
The following maps appear often in the paper:
\begin{equation}\label{pre-cplx}\textstyle \bigwedge^3F^*\xrightarrow{d_1} F^*\xrightarrow{d_0}F \xrightarrow{\delta_1}\bigwedge^3F,\end{equation}with
$d_1(\phi_3)=\xi(\phi_3)$, $d_0(\phi_1)=\phi_1(\xi)$, and
$\delta_1(f_1)=f_1\wedge \xi$, for $\phi_3\in \bigwedge^3F^*$, $\phi_1\in F^*$, and $f_1\in F$. Use Observation~\ref{doo-8.2}.\ref{doo-8.2.a} and (\ref{Gamma}) to see that 
$$(d_0\circ d_1)(\phi_3)=[\xi(\phi_3)](\xi)=\phi_3(\xi^{(2)})\quad\text{and}\quad 
(\delta_1\circ d_0)(\phi_1)=[\phi_1(\xi)]\wedge \xi=\phi_1(\xi^{(2)});$$so, in particular 
$A\otimes_R (\text{\rm\ref{pre-cplx}})$ is a complex. In (\ref{prove-3.2}) we prove that a modification of ${A\otimes_R\text{(\ref{pre-cplx})}}$ is exact and in Observation~\ref{yet-to-come} we prove that $A\otimes_R\text{(\ref{pre-cplx})}$ is exact.

If one uses the notation Remark~\ref{R2}.\ref{R2.z}, then the matrix for $d_0$ is $-\mathbf X$, the matrix for $d_1$ has $\goth f$ rows and  $\binom{\goth f}3$ columns and  the column corresponding to
$e_k^*\wedge e_j^*\wedge e_i^*$, for $1\le i<j<k\le \goth f$, is
\setcounter{MaxMatrixCols}{20}
$$\bmatrix 0&\dots&0&x_{j,k}&0&\dots&0&-x_{i,k}&0&\dots&0&x_{i,j}&0&\dots&0\endbmatrix^{\rm T},$$where the non-zero entries appear in positions $i$, $j$, and $k$, respectively; see Observation~\ref{doo-8.2}.\ref{doo-8.2.b} and Remark~\ref{R2}.\ref{R2.b}. (We use $M^{\rm T}$ to represent the transpose of the matrix $M$.) The matrix for $\delta_1$ is the transpose of the matrix for $d_1$. 
\end{remark}

\section{The  main ingredient.}\label{dream}

In this section we prove the following result.

\begin{lemma}\label{main-Dream-Lemma}
Adopt Data~{\rm\ref{data2}.\ref{data2-one}} and   Notation~{\rm\ref{Not2}.\ref{Not2.a}} with $4\le \goth f$.
If the base ring $R_0$ is an arbitrary commutative Noetherian ring, then the $R$-module $N$ is perfect of projective dimension $\binom {\goth f-2}2$. 
\end{lemma}

In Observation~\ref{yet-to-come} we  show that the module $N$ of Lemma~\ref{main-Dream-Lemma} is isomorphic to the module of (\ref{column}).

\begin{remark}
The assertion of Lemma~\ref{main-Dream-Lemma} does not hold for $\goth f=3$. Indeed, in the language of Remark~\ref{R2}, $N$, which is resolved by
  $$0\to R\xrightarrow{\bmatrix x_{2,3}&-x_{1,3}&x_{1,2}\endbmatrix^{\rm T}}R^3,$$ is not a perfect $R$-module and has projective dimension one, which is not equal to  $\binom {\goth f-2}2$.  (We use $M^{\rm T}$ to represent the transpose of the matrix $M$.)
\end{remark}

It is convenient to let 
\begin{align*}&\text{
$A'$ be the ring $A/(x_{1,2}\ ,\ x_{2,3}\ ,\ x_{1,3})$},
\end{align*} 
in the language of Remark~\ref{R2}.\ref{R2.z}. Our proof of Lemma~\ref{main-Dream-Lemma} is given in \ref{prove-dream} at the end of the section; it depends on Lemma~\ref{exact-seq} and on information about the rings $A$ and $A'$ which is contained in Lemma~\ref{KL}.

\begin{lemma}\label{exact-seq}
Adopt Data~{\rm\ref{data2}.\ref{data2-one}} and   Notation~{\rm\ref{Not2}.\ref{Not2.a}} with $3\le \goth f$.
If the base ring $R_0$ is a commutative Noetherian domain, then
there is an exact sequence of $A$-modules{\rm:}
\begin{equation}\label{exact-seq.1}0\to N\to A^3\to A\to A'\to 0.\end{equation}
In particular, if $(0)$ is the zero ideal of $A$, then $N_{(0)}$ is isomorphic to $A_{(0)}\oplus A_{(0)}$.
\end{lemma}

\begin{remarks}\begin{enumerate}[\rm(a)]
\item A strengthened version of Lemma~\ref{exact-seq} may be found in Proposition~\ref{improved}.
\item Lemma~{\rm\ref{exact-seq}} does hold when $\goth f=3$; indeed, 
(\ref{exact-seq.1}) becomes
$$0\to \frac{R^3}{\left(\bmatrix x_{2,3}\\-x_{1,3}\\x_{1,2}\endbmatrix\right)}\xrightarrow{\bmatrix 0&x_{1,2}&x_{1,3}\\-x_{1,2}&0&x_{2,3}\\-x_{1,3}&-x_{2,3}&0\endbmatrix} R^3\xrightarrow{\bmatrix x_{2,3}&-x_{1,3}&x_{1,2}\endbmatrix} R\to R_0\to 0,
$$which is exact.\end{enumerate}
\end{remarks}

\bigskip
The proof of Lemma~\ref{exact-seq} is given in \ref{prove-3.2}.

\begin{definition}\label{I-KL}Adopt the language of \ref{data2}.\ref{data2-one}, \ref{Not2}.\ref{Not2.a} and \ref{R2}.\ref{R2.z}. For each integer $\lambda$, between $1$ and $\goth f-1$, let $I_{\lambda}$ be the ideal $$I_\lambda=I+(\{x_{i,j}\mid 1\le i<j\le \lambda\})$$ of $R$. \end{definition}
\begin{example}\label{3.4}Retain the notation of Definition~{\rm\ref{I-KL}}. The ideal  $I_1$ is equal to $I$ (because the empty set generates the zero ideal) and the ideal $I_{\goth f-1}$ is equal to $(\{x_{i,j}\mid 1\le i<j\le \goth f-1\})$ (because $I$ is contained in the ideal $(\{x_{i,j}\mid 1\le i<j\le \goth f-1\})$. In particular, $A=R/I_1$ and $A'=R/I_3$.\end{example}
\begin{lemma}\label{KL} Adopt the language of {\rm \ref{data2}.\ref{data2-one}}, {\rm\ref{Not2}.\ref{Not2.a}}, {\rm \ref{R2}.\ref{R2.z}}, and {\rm\ref{I-KL}}. 
Let $\lambda$ be an integer between $1$ and $\goth f-1$. 
\begin{enumerate}[\rm(a)]\item\label{KL.a} If the base ring $R_0$ is an arbitrary commutative Noetherian ring, then $I_\lambda$ is a perfect ideal in $R$ of grade $\textstyle\binom{\goth f-2}2+\lambda-1$. In particular, if $4\le \goth f$, then $\grade I_3=\grade I_1+2$.
\item\label{KL.b} If the base ring $R_0$ is a commutative Noetherian domain, then $I_\lambda$ is a prime ideal.
\item\label{KL.c} If the base ring $R_0$ is an arbitrary commutative Noetherian ring, then  $I$ is a Gorenstein ideal in the sense of {\rm\ref{perfection}.\ref{gor}}. In particular, if $R_0$ is a Gorenstein ring, then $R/I$ is a Gorenstein ring. 
\end{enumerate}
\end{lemma}
\begin{remark}\label{require} The ``in particular assertion'' in (\ref{KL.a}) would be false if $\goth f$ were equal to $3$; because, in this case, $I_1$, which is equal to $(0)$, has grade $0$, and $I_3$, which is equal to $(x_{1,2}\ ,\ x_{1,3}\ ,\ x_{2,3})$, has grade $3$. Of course,  the parameter $\lambda$, which is assumed to be at most $\goth f-1$, is not permitted to be $3$, when $\goth f=3$.
\end{remark}
\begin{proof}(\ref{KL.a},\ref{KL.b}) The ideal $I_\lambda$ is equal to the ideal $\Pf(X;\lambda;\lambda)$ of \cite{KL80}. The assertion follows from  \cite[Thm.~12]{KL80}. The statement of \cite[Thm.~12]{KL80} only considers the case where $R_0$ is a domain; however,  as soon as one knows that $I_\lambda$ is a perfect ideal when $R_0$ is equal to the ring of integers and when $R_0$ is equal to a field, then $I_\lambda$ built with $R_0=\mathbb Z$ is a generically perfect ideal and consequently $I_\lambda$ built over an arbitrary commutative Noetherian $R_0$ is a perfect ideal; see, for example \cite[Prop.~3.2 and Thm.~   3.3]{BV}. 

\medskip\noindent(\ref{KL.c}) A proof that $R/I$ is a Gorenstein ring whenever $R_0$ is Gorenstein is given in \cite[Thm.~17]{KL80}. A more explicit statement and proof of this result is given in \cite[Corollary]{A79}. In particular, when $R_0$ 
is equal to the ring of integers, then there exists a resolution $\mathbb F$ of $R/I$ by free $R$-modules which has the property that the length of $\mathbb F$ is $\binom{\goth f-2}2$ and the free module of $\mathbb F$ in position $\binom{\goth f-2}2$ has rank one. Now let $R_0$ be an arbitrary commutative Noetherian ring. We explained in the proof of (\ref{KL.a}) and (\ref{KL.b}) that $I$ is a perfect ideal in $R$. The ``Persistence of Perfection Principle'', Theorem~\ref{PoPP}, now guarantees that the back Betti number in a resolution of $R/I$ by free $R$-modules is one; and therefore, $I$ is a Gorenstein ideal in the sense of  \ref{perfection}.\ref{gor}.
\end{proof}
\begin{remark}An alternate phrasing of the proof of Lemma~\ref{KL}, parts (\ref{KL.a}) and (\ref{KL.b}), (but really the same argument in a different form) involves the Grassmannian $\operatorname{Gr}(2,\goth f)$ of rank $2$ free summands of the rank $\goth f$ free $R_0$-module $V$. The ideal $I$ is the ideal of ``quadratic relations'' which define the homogeneous coordinate ring of the image of the Pl\"ucker embedding of $\operatorname{Gr}(2,\goth f)$ into $\mathbb P(\bigwedge^2 V)$. The ideal $I_\lambda$ defines the homogeneous coordinate ring of the Pl\"ucker embedding of the  Schubert subvariety $\Omega(\goth f-\lambda,\goth f)$ of  $\operatorname{Gr}(2,\goth f)$. The  Schubert subvariety ${\Omega(\goth f-\lambda,\goth f)}$ consists of all $W$ in $\operatorname{Gr}(2,\goth f)$ such that $i\le \rank (W\cap V_i)$ for the flag $V_1\subsetneq V_2$ where $V_1$ is the the summand of $V$ with basis $e_{\lambda+1},\dots,e_\goth f$ and $V_2=V$. The original proofs that the homogeneous coordinate rings of the Schubert subvarieties of the Grassmannian are Cohen-Macaulay domains are \cite[Thm.~3.1$^*$, (3.10), Cor.~4.2]{H73}, \cite[Thm.~1]{L72}, and \cite[Thm.~II.4.1 and Thm.~III.4.1]{M72}. A version which contains many details is \cite[Thm.~1.4, the bottom of page 52, Cor.~5.18, Thm.~6.3]{BV}.
\end{remark}

One consequence of Lemma~\ref{KL} is that $I$ is grade unmixed. This fact facilitates the identification of regular elements in $A$. Corollary~\ref{KL-consq} and its style of proof are used in the proof of Corollary~\ref{corollary}.\ref{cor-a}.

\begin{corollary}\label{KL-consq} Adopt the language of {\rm \ref{data2}.\ref{data2-one}}, {\rm\ref{Not2}.\ref{Not2.a}}, and {\rm \ref{R2}.\ref{R2.z}}. If the base ring $R_0$ is an arbitrary commutative Noetherian ring, then $x_{1,2}\ ,\ x_{1,3}$ is a regular sequence on $A$.
\end{corollary}

\begin{proof}
Every associated prime $P$ of $R/I$ in $R$ has grade $PR_P=\binom{\goth f-2}2$. 
Lemma~\ref{KL}.\ref{KL.a} assures that
$I_2$, which equals $(I\ ,\ x_{1,2})$, is a perfect ideal of grade $\binom{\goth f-2}2+1$ in $R$; hence $x_{1,2}$ is not in any associated prime of $R/I$ (that is, $x_{1,2}$ is regular on $R/I$)  and every associated prime $P$ of $R/(I\ ,\ x_{1,2})$ in $R$ has grade $PR_P=\binom{\goth f-2}2+1$. We prove that $x_{1,3}$ is regular on $R/(I\ ,\ x_{1,2})$ by showing that $\textstyle\binom{\goth f-2}2+2\le \grade PR_P$ for all primes $P$ of $R$ which contain $(I\ ,\ x_{1,2}\ ,\ x_{1,3})$. Let $P$ be such a prime. Consider the Pfaffian
$$
x_{1,2}x_{3,j}-x_{1,3}x_{2,j}+x_{1,j}x_{2,3}\in I\subseteq P.$$Thus, $x_{2,3}x_{1,j}$ is in $P$ for $3\le j\le \goth f$. It follows that  either $I_3$, which is  $(I\ ,\  x_{1,2}\ ,\ x_{1,3}\ ,\ x_{2,3})$, is contained in $P$ or $(I\ ,\ x_{1,2}\ ,\ x_{1,3}\ ,\ \dots\ ,\ x_{1,\goth f})\subseteq P$. Lemma~\ref{KL}.\ref{KL.a} ensures that $I_3$ has grade ${\binom{\goth f-2}2+2}$. The ideal $(I\ ,\ x_{1,2}\ ,\ x_{1,3}\ ,\ \dots\ ,\ x_{1,\goth f})$ is equal to $\Pf_4(\mathbf X')$ plus an ideal generated by $\goth f-1$ indeterminates, where $\mathbf X'$ is $\mathbf X$ with row and column $1$ deleted. Thus $(I\ ,\ x_{1,2}\ ,\ x_{1,3}\ ,\ \dots\ ,\ x_{1,\goth f})$ has grade 
$$\textstyle \binom{\goth f-3}2+\goth f-1=\binom{\goth f-2}2+ 2.$$ In either event, $\binom {\goth f-2}2+2\le \grade P$ and the proof is complete.
\end{proof}

\begin{chunk}\label{prove-3.2} {\bf Proof of Lemma \ref{exact-seq}.}
We prove that 
\begin{equation}\label{key}\textstyle \bigwedge^3\overline{F}^*\xrightarrow{\ \ \overline{d_1}\ \ }\overline{F}^*
\xrightarrow{\ \ d_0'\ \ }
A^3\xrightarrow{\ \ \rho\ \ }A\to A'\to 0 \end{equation}is an exact sequence of $A$-modules, where
$d_1:\bigwedge^3F^*\to F^*$ is $d_1(\phi_3)=\xi(\phi_3)$, as given in 
Notation~\ref{Not2}.\ref{Not2.a.iii} and Remark~\ref{promise}, 
$d_0'$ is the composition
 $$\textstyle \overline{F}^*\xrightarrow{\ \ \overline{d_0}\ \ }\overline{F}=\bigoplus\limits_{i=1}^\goth f Ae_i
\xrightarrow{\ \ \text{projection}\ \ }\bigoplus\limits_{i=1}^3 Ae_i, 
$$  where $d_0:F^*\to F$ is
$d_0(\phi_1)=\phi_1(\xi)$ as described in Remark~\ref{R2}.\ref{R2.e} and Remark~\ref{promise}, and $\rho$
is given by the matrix \begin{equation}\label{delta1}\rho=\bmatrix x_{2,3}&-x_{1,3}&x_{1,2}\endbmatrix.\end{equation} (The basis $e_1,\dots,e_\goth f$ for $F$ is introduced in Remark~\ref{R2}.)
Once we show that (\ref{key}) is an exact sequence, then the proof is complete. Indeed,  
\begin{align}&N=\coker \overline{d_1}\cong \im d_0'=\ker \rho;
\quad\text{hence,}\label{3.10.3}\\
&0\to N\to
A^3\xrightarrow{\ \ \rho\ \ }A\to A'\to 0\notag\end{align} is exact, as claimed in (\ref{exact-seq.1}).

We first show that (\ref{key}) is a complex. To show that $d_0'\circ \overline{d_1}=0$ it suffices to show that the image of $d_0\circ d_1$ is contained in $I\cdot F$ and this was done in Remark~\ref{promise}.   The matrix for $\rho$ is given in (\ref{delta1}) and the matrix 
\begin{equation}\label{d0'}d_0'=-\bmatrix 0&x_{1,2}&x_{1,3}&x_{1,4}&\dots&x_{1,\goth f}\\ 
-x_{1,2}&0&x_{2,3}&x_{2,4}&\dots&x_{2,\goth f}\\
-x_{1,3}&-x_{2,3}&0&x_{3,4}&\dots&x_{3,\goth f}\endbmatrix
\end{equation} for 
$d_0'$ 
may be read 
from the discussion in Remark~\ref{promise}. 
It is clear that $\rho\circ d_0'=0$ and that the complex (\ref{exact-seq.1}) is exact at $A$ and $A'$. We next show that (\ref{exact-seq.1}) is exact at $A^3$. 
Suppose 
$$\alpha=\bmatrix a_1&a_2&a_3\endbmatrix^{\rm T}$$ is an element of $A^3$ with
$\rho(\alpha)=0$ in $A$. (We use $M^{\rm T}$ to represent the transpose of the matrix $M$.)
In other words, $x_{2,3}a_1-x_{1,3}a_2+x_{1,2}a_3=0$ in $A$. In particular,
$$x_{2,3}a_1\in (x_{1,2}\ ,\ x_{1,3})A\subseteq (x_{1,2}\ ,\ x_{1,3}\ ,\ \dots\ ,\ x_{1,\goth f})A.$$ The ideal $(x_{1,2}\ ,\ x_{1,3}\ ,\ \dots\ ,\ x_{1,\goth f})A$ of $A$ is prime; indeed,
$$(x_{1,2}\ ,\ x_{1,3}\ ,\ \dots\ ,\ x_{1,\goth f})+I=(x_{1,2}\ ,\ x_{1,3}\ ,\dots\ ,\ x_{1,\goth f})+\Pf_4(\mathbf X'),$$ where ${\mathbf X}'$ is the matrix $\mathbf X$ of Remark~\ref{R2}.\ref{R2.e} with row one and column one deleted. The matrix $\mathbf X'$ is a generic alternating matrix which does not involve the variables $x_{1,2}\ ,\ \dots\ ,\ x_{1,\goth f}$; so \cite[Thm.~12]{KL80} guarantees that $\Pf_4(\mathbf X')$ is prime; see, for example Lemma~\ref{KL}. 

The product $x_{2,3}a_1$ is in the prime ideal $(x_{1,2}\ ,\ \dots\ ,\ x_{1,\goth f})A$ and $x_{2,3}\notin(x_{1,2}\ ,\ \dots\ ,\ x_{1,\goth f})A$; thus, $a_1\in (x_{1,2}\ ,\ \dots\ ,\ x_{1,\goth f})A$ and a quick glance at (\ref{d0'}) shows that there is an element $\overline{\phi_1}$ in $\overline{F}$ such that 
$$\alpha -d_0'(\overline{\phi_1})=\bmatrix 0&a_2'&a_3'\endbmatrix^{\rm T},$$  
for  some $a_2'$ and $a_3'$ in $A$.  The equation $-x_{1,3}a_2'+x_{1,2}a_3'=0$ in $A$ shows that $x_{1,3}a_2'$ is an element of the prime ideal 
$(x_{1,2})A=I_2A$; see Lemma~\ref{KL}. Hence, $a_2'$ is in $(x_{1,2})A$ and 
a further modification $\alpha -d_0'(\overline{\phi_1})$ by a boundary which only involves the first column of $d_0'$ yields an element of the kernel of $\rho$ of the form $\bmatrix 0&0&a_3''\endbmatrix^{\rm T}$. The element $a_3''$  is zero because $A$ is a domain; and therefore, $\alpha\in \im d_0'$.

The argument that (\ref{exact-seq.1}) is exact at $\overline{F}^*$ is very similar to the preceding argument. Suppose $\alpha=[a_1\ ,\ \dots\ ,\ a_{\goth f}]^{\rm T}$ is an element of $\ker d_0'$. The third row of the equation $d_0'\alpha=0$ yields that $x_{3,\goth f}a_f$ is an element of the prime ideal $I_{\goth f-1}A$, in the language of Definition~\ref{I-KL} and Lemma~\ref{KL}; but $x_{3,\goth f}\notin I_{\goth f-1}A$; so $a_{\goth f}\in I_{\goth f-1}$. On the other hand, for each $x_{i,j}\in I_{\goth f-1}$, $$\overline{d_1}(e_\goth f^*\wedge e_j^*\wedge e_i^*)=x_{j,\goth f}e_i^*-x_{i,\goth f}e_j^*+x_{i,j}e_\goth f^*;$$ hence  there is an element $\overline{\phi_3}\in \bigwedge^3\overline{F}^*$ so that $$\alpha -\overline {d_1}(\overline{\phi_3})=[a_1'\ ,\ \dots\ ,\ a_{\goth f-1}'\ ,\ 0]^{\rm T}.$$ The third row of the equation $d_0'(\alpha -\overline {d_1}(\overline{\phi_3}))=0$ yields that $x_{3,\goth f-1}a_{\goth f-1}'\in I_{\goth f-2}A$.
Use elements of the form $\overline{d_1}(e_{\goth f-1}^*\wedge e_j^*\wedge e_i^*)$ to remove $a_{\goth f-1}'$ (while keeping $0$ in the bottom position).
Continue in this manner to find $\overline{\phi_3}^{\,\dagger}\in \bigwedge^3\overline{F}^*$ so that $$\alpha -\overline {d_1}(\overline{\phi_3}^{\,\dagger})=[a_1^\dagger\ ,\ a_2^\dagger\ ,\ a_3^\dagger\ ,\ 0\ ,\ \dots,0]^{\rm T}.$$The second equation of 
$$\bmatrix 0&x_{1,2}&x_{1,3}\\-x_{1,2}&0&x_{2,3}\\-x_{1,3}&-x_{2,3}&0\endbmatrix \bmatrix a_1^\dagger\\[3pt]a_2^\dagger\\[3pt]a_3^\dagger\endbmatrix =d_0'(\alpha -\overline {d_1}(\overline{\phi_3}^\dagger)=0$$ yields $a_3^\dagger\in (x_{1,2})A$; hence there exists $\overline{\phi_3}^{\,\ddag}\in \bigwedge^3\overline{F}^*$, so that $$\alpha -\overline {d_1}(\overline{\phi_3}^{\,\ddag})=[a_1^\ddag\ ,\ a_2^\ddag\ ,\ 0\ ,\ \dots\ ,\ 0]^{\rm T}.$$Now one sees that $x_{1,2}a_1^\ddag=x_{1,2}a_2^\ddag=0$ in the domain $A$; hence $a_1^\ddag=a_2^\ddag=0$, $\alpha$ is a boundary, (\ref{exact-seq.1}) is exact.

The final assertion, that $N$ has rank two as an $A$-module, is an immediate consequence of the exactness of (\ref{exact-seq.1}). Indeed, $A$ is a domain (see \cite[Thm.~12]{KL80} or Lemma~\ref{KL}) and $A'_{(0)}=0$. 
\qed
\end{chunk}

\begin{chunk}\label{prove-dream}{\bf{Proof of Lemma \ref{main-Dream-Lemma}.}} The module $N$, built over an arbitrary ring $R_0$, is obtained from the module $N$, built over the ring of integers $\mathbb Z$, by way of the base change $R_0\otimes_{\mathbb Z}-$. According to the theory of generic perfection (see, for example \cite[Prop.~3.2 and Thm.~3.3]{BV}) in order to prove that $N$, built over an arbitrary ring $R_0$, is a perfect $R$-module,  it suffices to prove that $N$ is a  perfect $R$-module 
when $R_0=\mathbb Z$ and when $R_0$ is a field. Fix one of these choices for $R_0$ and consider the exact sequence of Lemma~\ref{exact-seq}.

It was observed in Example~\ref{3.4} that $A=R/I_1$ and $A'=R/I_3$; consequently, Lemma~\ref{KL}.\ref{KL.a} guarantees that $A$ and $A'$ are perfect $R$-modules and  $\pd_R A'=\pd_R A+2$. (This is where the hypothesis $4\le \goth f$ is required; see Remark~\ref{require}.)
Let $P$ be a prime ideal of $R$ which is in the support of $N$. Lemma~\ref{exact-seq} shows that the module $N$ embeds into a free $A$-module; hence, $P$ is in the support of $A$ and $A_P$ is a Cohen-Macaulay ring.  The localization  $A'_P$ is either zero or 
a Cohen-Macaulay ring with $\dim A'_P=\dim A_P-2$. In either event, 
 we apply the usual argument about the growth of depth in an exact sequence 
(see, for example, \cite[Prop.~1.2.9]{BH}),  to
 the localization of the exact sequence (\ref{exact-seq.1}) at $P$ in order to conclude  that  $\depth A_P\le \depth N_P$. At this point the inequalities
\begin{equation}\label{3.11.0}\depth N_P\le \dim N_P\le_* \dim A_P=\depth A_P\le \depth N_P\end{equation} all hold; consequently, equality holds throughout. (The inequality labeled * holds because $N_P$ is an $A_P$-module.) Thus, $N_P$ is a Cohen-Macaulay $R_P$-module and 
\begin{equation}\label{just-above}\textstyle\pd_{R_P}N_P=\pd_{R_P}A_P=\pd_RA=\binom {\goth f-2}2.\end{equation}
(The first equality is a consequence of the Auslander-Buchsbaum theorem; the second equality  is explained in \ref{perfection}.\ref{constant-pd}; and the third equality is a consequence of Lemma~\ref{KL}.) Thus, $N$ is a perfect $R$-module of projective dimension $\binom{\goth f-2}2$ (see \ref{perfection}.\ref{perfection.d}, if necessary) and the proof is complete. 
\qed\end{chunk}

Section~\ref{main} is concerned with the ring $\mathcal R$ of Data~{\rm\ref{data2}} and   Notation~{\rm\ref{Not2}}. The ring $\mathcal R$ is a polynomial ring over $R$ and the $\mathcal R$-modules $\mathcal N=\mathcal R\otimes_RN$, 
$\mathcal A=\mathcal R\otimes_RA$, and $\mathcal F=\mathcal R\otimes_RF$ are obtained from the corresponding $R$-modules by way of a base change. It is convenient to record the results of the present section in the language of the future section. 
\begin{corollary}\label{carry forward}
Adopt Data~{\rm\ref{data2}} and   Notation~{\rm\ref{Not2}} with $4\le \goth f$.
If the base ring $R_0$ is an arbitrary commutative Noetherian ring, then the $\mathcal R$-modules $\mathcal N$ and $\mathcal A$ are perfect of projective dimension $\binom {\goth f-2}2${$;$} furthermore $I\mathcal R$ is a Gorenstein ideal. 
\end{corollary}
\begin{proof} Apply Lemmas~\ref{main-Dream-Lemma} and \ref{KL}. \end{proof}

We close this section by redeeming 
assorted 
promises. Assertion~(\ref{yet.a}) was promised in 
Remark~\ref{promise}.
Assertion~(\ref{yet.b}) was promised in the introduction when we claimed that Section~\ref{dream} is about the image of $\overline{d_0}$; however, until this point, it appears that Section~\ref{dream} is about $N$, which is the cokernel of $\overline{d_1}$.
The homological properties of $N$, which are listed in (\ref{yet.c}) and (\ref{yet.d}), were also promised in the introduction.
\begin{observation}\label{yet-to-come}
Adopt the language of {\rm \ref{data2}.\ref{data2-one}}, {\rm\ref{Not2}.\ref{Not2.a}}, and {\rm \ref{R2}.\ref{R2.z}}. Assume that $R_0$ is a domain.
\begin{enumerate}[\rm(a)]
\item\label{yet.a} The complex $A\otimes_R\text{\rm(\ref{pre-cplx})}$ is exact. \item\label{yet.b} The module $N$ {\rm (}of Lemma {\rm\ref{main-Dream-Lemma}} and elsewhere{\rm)} is isomorphic to the module of {\rm (\ref{column})}. 
\item \label{yet.c} The $A$-module $N$ is self-dual.
\item \label{yet.d} If $R_0$ is a Cohen-Macaulay domain, then   
$N$ is a self-dual maximal Cohen-Macaulay $A$-module of rank two.
\item \label{yet.e} 
If $R_0$ is a Gorenstein domain, and 
$$\mathbb X:\quad \cdots \xrightarrow{d_4}\mathbb X_3\xrightarrow{d_3}\mathbb X_2\xrightarrow{d_2} 
\textstyle \bigwedge^3\overline{F}^*\xrightarrow{\overline{d_1}} \overline{F}^*
$$
is a resolution of $N$ by free $A$-modules, then 
$$\textstyle \mathbb Y:\quad \cdots  \xrightarrow{d_4}\mathbb X_3\xrightarrow{d_3}\mathbb X_2\xrightarrow{d_2}\bigwedge^3\overline{F}^*\xrightarrow{\overline{d_1}} \overline{F}^*\xrightarrow{\overline{d_0}}\overline{F} \xrightarrow{\overline{\delta_1}}\bigwedge^3\overline{F}\xrightarrow{d_2^*}\mathbb X_2^*\xrightarrow{d_3^*}\mathbb X_3^*\xrightarrow{d_4^*}\cdots$$
is a self-dual totally acyclic complex.  {\rm(}In other words, $\HH_\bullet(\mathbb Y)=\HH_{\bullet}(\mathbb Y^*)=0$ and, after making the appropriate shift, $\mathbb Y^*$ is isomorphic to $\mathbb Y$.{\rm)}
\end{enumerate}\end{observation}

\begin{proof} (\ref{yet.a}) We are supposed to prove that the complex \begin{equation}\label{name-me}\textstyle \bigwedge^3\overline{F}^*\xrightarrow{\overline{d_1}} \overline{F}^*\xrightarrow{\overline{d_0}}\overline{F} \xrightarrow{\overline{\delta_1}}\bigwedge^3\overline{F}\end{equation} is exact. (Recall from \ref{Not2}.\ref{overline} that $\overline{\phantom{x}}$ is the functor $A\otimes_R-$.) We showed in (\ref{key}) that $$\textstyle \bigwedge^3\overline{F}^*\xrightarrow{\overline{d_1}} \overline{F}^*\xrightarrow{\operatorname{projection}\circ\overline{d_0}}A^3$$is exact. It follows that
$$\im \overline{d_1}\subseteq \ker \overline{d_0}\subseteq \ker(\operatorname{projection}\circ\overline{d_0})= \im \overline{d_1}$$ and (\ref{name-me}) is exact at $\overline{F}^*$. 

We now prove that (\ref{name-me}) is exact at $\overline{F}$. Let $f_1=\sum_{i=1}^\goth f a_ie_i$ be in $\ker \overline{\delta_1}$, with $a_i\in A_i$ and $e_1,\dots, e_n$ a basis for $\overline{F}$. Use the  coefficient of $e_1\wedge e_i\wedge e_j$ in $0=\overline{\delta_1}(f_1)$ in order to see that
$$x_{i,j}a_1\in(x_{1,i}\ ,\ x_{1,j})\subseteq (x_{1,2}\ ,\ x_{1,3}\ ,\ \dots\ ,\ x_{1,\goth f})$$
for all $i$ and  $j$ with $2\le i<j\le \goth f$. The ideal $(x_{1,2}\ ,\ x_{1,3}\ ,\ \dots\ ,\ x_{1,\goth f})$ of $A$ is prime (indeed, $A/(x_{1,2}\ ,\ x_{1,3}\ ,\ \dots\ ,\ x_{1,\goth f})$ is the domain defined by ``$\Pf_4$'' of a smaller generic matrix) and $x_{i,j}$ is not in  $(x_{1,2}\ ,\ x_{1,3}\ ,\ \dots\ ,\ x_{1,\goth f})$. Therefore, $a_1\in (x_{1,2}\ ,\ x_{1,3}\ ,\ \dots\ ,\ x_{1,\goth f})$ and there is an element $\phi_1\in \overline{F}^*$ with $f_1^\dagger=f_1-\overline{d_0}(\phi_1)=\sum_{i=2}^\goth f a_i^\dagger e_i$. (Recall that $-\mathbf X$ is the matrix for $d_0$.) The coefficient of $e_1\wedge e_2\wedge e_3$ in $0=\overline{\delta_1}(f_1^\dagger)$ shows $a_2^\dagger x_{1,3}$ is in the prime ideal $(x_{1,2})$; hence, $a_2^\dagger\in (x_{1,2})$ and one may use the first column of $\mathbf X$ to remove $a_2^\dagger$ without damaging $a_1^\dagger=0$. In other words, there exists $\phi_1^\ddag\in \overline{F}^*$ with $f_1^\ddag=f_1-d_0(\phi_1^\ddag)=\sum_{i=3}^\goth f a_ie_i$. The coefficient of $e_1\wedge e_2\wedge e_j$ in $0=\overline{\delta_1}(f_1^\ddag)$ shows that $x_{1,2}a_j^\ddag=0$ for $3\le j\le \goth f$. Hence,  $a_j^\ddag=0$ for $3\le j\le \goth f$,  $f_1$ is a boundary in (\ref{name-me}), and (\ref{name-me}) is exact.

\medskip(\ref{yet.b}) Apply (\ref{yet.a}) to see that $N=\coker \overline{d_1}\cong \im \overline{d_0}={\rm (\ref{column})}$. 

\medskip(\ref{yet.c}) The definition $N=\coker \overline{d_1}$ guarantees that 
$\textstyle\bigwedge^3\overline{F}^*\xrightarrow{\overline{d_1}} \overline{F}^*
\to N\to 0$ is exact. Apply $\Hom_A(-,A)$ to learn that 
$$
0\to N^*\to \overline{F}^{**}\xrightarrow{\overline{d_1}^*} \textstyle\bigwedge^3\overline{F}^{**}$$is exact. It is easy to see that 
$\overline{F}^{**}\xrightarrow{\overline{d_1}^*} \bigwedge^3\overline{F}^{**}$ 
is isomorphic to 
$\overline{F}\xrightarrow{\overline{\delta_1}} \bigwedge^3\overline{F}$.
Assertion (\ref{yet.a}) now gives that $N\cong \ker  \overline{\delta_1}\cong \ker \overline{d_1}^*\cong N^*$.

\medskip(\ref{yet.d}) Lemma~\ref{main-Dream-Lemma}, especially (\ref{3.11.0}) ensures that $N$ is a maximal Cohen-Macaulay $A$-module. The rank of $N$ is calculated in Lemma~\ref{exact-seq}. The self-duality of $N$ is established in (\ref{yet.c}).

\medskip(\ref{yet.e}) It follows from local duality (or the Auslander-Bridger formula, see, for example, \cite[Thms.~1.4.8 and 1.4.9]{Ch00}) that the maximal Cohen-Macaulay module $N$ over the Gorenstein ring $A$ satisfies $\Ext^i_A(N,A)=0$ for all positive $i$. So $\mathbb X\to N\to 0$ and $0\to N^*\to \mathbb X^*$ are both acyclic. The complexes $\mathbb X$ and $\mathbb X^*$  may be patched together at $N\cong N^*$ to form the totally acyclic complex $\mathbb Y$.
\end{proof}

\bigskip

\section{The  main result.}\label{main}

\bigskip

The main result of the paper is Theorem~\ref{main-Theorem} where we prove that $J$ is a perfect Gorenstein ideal of grade $\binom{\goth f-2}2+2$. We estimate the grade of $J$ in Lemma~\ref{*.enough} and we use the exact sequence (\ref{*.claim.1}) to estimate the projective dimension of $\mathcal R/J$.

\begin{lemma}\label{*.enough}Adopt the language of {\rm\ref{data2}} and 
{\rm\ref{Not2}} with $3\le \goth f$. If the base ring $R_0$ is 
an arbitrary commutative Noetherian ring, then the height of the ideal $J$ satisfies the inequality $$\textstyle \binom{\goth f-2}2+2\le \htt J.$$
\end{lemma}

\begin{remark}\label{4.2} The assertion of Lemma~\ref{*.enough} is false when $\goth f=2$ because in this case $J$ equals $(t_1x_{1,2}\ ,\ t_2x_{1,2})$, which has height $1$; see Remark~\ref{R4}.\ref{funny} for a continuation of this example. On the other hand, Lemma~\ref{*.enough} does hold when $\goth f=3$; indeed, in this case, $J$ is the ideal generated by the maximal minors of the generic matrix
$$\bmatrix t_1&t_2&t_3\\x_{2,3}&-x_{1,3}&x_{1,2}\endbmatrix;$$see Remark~\ref{R4}.\ref{funny+} for a continuation of this example.
\end{remark}

\begin{proof} It suffices to replace $R_0$ with $R_0/p$ for some minimal prime ideal $p$ in $R_0$ and to prove the result when $R_0$ is a domain.
We use the language of Remark~\ref{R2} and view $J$ as the ideal $\Pf_4(\mathbf {X})+
I_1(\mathbf {tX})$ in the ring $\mathcal R=R_0[\{x_{i,j}\},\{t_i\}]$. Let $P$ be a prime ideal of $\mathcal R$ which contains $J$. We show $$\textstyle\binom{\goth f-2}2+2\le \htt P.$$

If $t_1\in P$, then $I'=\Pf_4(\mathbf{X})+(t_1)$ is a prime ideal of height $\binom{\goth f-2}2+1$ which is contained in $P$; furthermore, the first entry of $\mathbf {tX}$ is a  non-zero element of $P\setminus I'$. Thus, 
$\binom{\goth f-2}2+2\le \htt P$. 

If $t_1\notin P$, then let 
${\mathbf X}'$ be $\mathbf X$  with the first column removed, 
${\mathbf X}''$ be $\mathbf X$   with the first row and first column removed,
and $I''$ be the ideal $\Pf_4({\mathbf X}'')$. 
Observe that $I''$ is a  
prime ideal of height $\binom{\goth f-3}2$ (this is where we use the hypothesis that $3\le \goth f$);  $I''$ is contained in $P$;
and the entries of $\mathbf {tX}'$ form a regular sequence on $\mathcal R_{\ t_1}/I''\mathcal R_{\ t_1}$ in $P\mathcal R_{\ t_1}$. It follows that
$$\textstyle 
\binom{\goth f-2}2+2=\binom{\goth f-3}2+\goth f-1\le \htt P \mathcal R_{\ t_1}=\htt P.
\vspace{-18pt}$$
\end{proof}

\begin{proposition}\label{*.claim}Adopt the language of {\rm\ref{data2}} and {\rm\ref{Not2}}. If $2\le \goth f$ and $R_0$ is a Cohen-Macaulay domain,
 then  there is an exact sequence of $\mathcal A$-modules{\rm:}
\begin{equation}\label{*.claim.1}0\to \mathcal A\xrightarrow{\tau}
\mathcal N\xrightarrow{\tau(\xi)}\mathcal A\to \mathcal R/ J\to 0.\end{equation} The map
$\tau:\mathcal A\to  \mathcal N$ sends the element $1$ of $\mathcal A$ to the class of $\tau$ in $$\textstyle \mathcal N
=
\mathcal A\otimes_{\mathcal R}\coker\big(\xi:\bigwedge^3\mathcal F^*\to \mathcal F^*\big).$$ If $\phi_1$ is in $\mathcal F^*$, then the map $\tau(\xi):
\mathcal N\to \mathcal A$ sends the class of $\phi_1$ in $
\mathcal N$ to the class of $[\tau(\xi)](\phi_1)$ in $\mathcal A=\mathcal R/(I\cdot \mathcal R)$. The map $\mathcal A\to \mathcal R/ J$ is the natural quotient map $$\mathcal A=\mathcal R/(I\cdot \mathcal R)\to \mathcal R/(I\cdot \mathcal R+ K)=\mathcal R/J.$$ 
\end{proposition}

\begin{remarks}\label{R4}\begin{enumerate}[\rm(a)]\item After we prove Theorem~\ref{main-Theorem}, we are able to improve Proposition~\ref{*.claim}. In the improved version, $R_0$ is allowed to be an arbitrary commutative Noetherian ring. See Proposition~\ref{improved}.  
\item\label{here it is} The exact sequence $0\to \mathcal A\to \mathcal N\to J\mathcal A \to 0$, which is a consequence of   (\ref{*.claim.1}), exhibits $J\mathcal A$ as a Bourbaki ideal of $\mathcal N$, in the sense of \cite{A66, M80, BHU87, SUV03}.
\item The map $\tau({\xi})$ of (\ref{*.claim.1}) is well-defined.
Indeed,
if $\phi_3\in \bigwedge^3\mathcal F^*$, then $\xi(\phi_3)$ represents $0$ in $\mathcal N$ and
$[\tau(\xi)](\xi(\phi_3))$, which is equal to $\xi^{(2)}(\phi_3\wedge \tau)$ by (\ref{Gamma}) and (\ref{compat}), is equal to $0$ in $\mathcal A$. 
\item It is not difficult to see that (\ref{*.claim.1}) is a complex of $\mathcal A$-modules.
\item \label{funny}If $\goth f=2$, then $\mathcal R=\mathcal A$ and, in the language of Remark~\ref{R2}, the complex (\ref{*.claim.1}) is
$$0\to\mathcal R \xrightarrow{\bmatrix t_1\\t_2\endbmatrix}\mathcal R^2\xrightarrow{\bmatrix -t_2x_{1,2}&t_1x_{1,2}\endbmatrix}\mathcal R\to \mathcal R/(t_1x_{1,2}\ ,\ t_2x_{1,2})\to 0,$$ which is exact, see Remark~\ref{4.2}
\item \label{funny+}If $\goth f=3$, then $\mathcal R=\mathcal A$ and, in the language of Remark~\ref{R2}, the complex (\ref{*.claim.1}) is
\begin{align*}0&\to\mathcal R \xrightarrow{\bmatrix t_1\\t_2\\t_3\endbmatrix}\frac{\mathcal R^3}{\left(\bmatrix x_{2,3}\\-x_{1,3}\\x_{1,2}\endbmatrix\right)}\xrightarrow{\bmatrix -t_2x_{1,2}-t_3x_{1,3}&t_1x_{1,2}-t_3x_{2,3}&t_{1}x_{1,3}+t_2x_{2,3}\endbmatrix}\mathcal R\\&\to \mathcal R/(-t_2x_{1,2}-t_3x_{1,3}\ ,\ t_1x_{1,2}-t_3x_{2,3}\ ,\ t_{1}x_{1,3}+t_2x_{2,3})\to 0,\end{align*} which is exact; see Remark~\ref{4.2}\end{enumerate}
\end{remarks}

\medskip\noindent   
Observation  
\ref{*.10.9}
and Lemma \ref{*.critical} are used in the proof of Proposition~\ref{*.claim}, which is given in \ref{*.proof-of-claim}.

\begin{observation}\label{*.10.9} Retain the hypotheses of Proposition~{\rm\ref{*.claim}}. The complex {\rm(\ref{*.claim.1})} is exact at $\mathcal R/J$ and at both copies of $\mathcal A$.\end{observation}
\begin{proof} It is clear that (\ref{*.claim.1}) is exact at $\mathcal R/J$ and at the right hand $\mathcal A$. We  prove that (\ref{*.claim.1}) is exact at the left hand $\mathcal A$. Let $r\in \mathcal R$ with $r\cdot \tau\equiv\xi(\phi_3)\mod I\mathcal F$ for some $\phi_3\in \bigwedge^3 \mathcal F^*$. Apply $r\tau$ to $\xi$ and use Observation~\ref{doo-8.2}.\ref{doo-8.2.a} to learn that 
$$r\cdot \tau(\xi)\equiv[\xi(\phi_3)](\xi)\equiv\phi_3(\xi^{(2)})\in I\mathcal F.$$It follows that $r\cdot K\subseteq I$. The ideal $I$ is prime and degree considerations show that $K\not\subseteq I$. It follows that $r\in I$. Thus, $\tau: \mathcal A\to \mathcal N$ is an injection.
\end{proof}

\begin{lemma}\label{*.critical}Adopt the language of {\rm\ref{data2}} and {\rm\ref{Not2}}. 
Let  $\phi_1,\phi_1'$ be elements of $\mathcal F^*$ with the property that the element $\phi_1\wedge \phi_1'$ is part of a basis for $\mathcal F^*$ and let $x$ be the element $\xi(\phi_1\wedge \phi_1')$ of $\mathcal R$. Then the following statements hold. \begin{enumerate}[\rm(a)] \item\label{*.critical.a} If the base ring $R_0$ is a commutative Noetherian domain, then the localization {\rm(\ref{*.claim.1})}$_x$ of the complex {\rm(\ref{*.claim.1})} at $x$  is isomorphic to
\begin{align*}0\to \mathcal A_{\, x}\xrightarrow {\bmatrix -[\tau(\xi)](\phi_1')\\\phantom{-}[\tau(\xi)](\phi_1)\endbmatrix} \mathcal A_{\,x} \oplus \mathcal A_{\,x} &\xrightarrow {\bmatrix [\tau(\xi)](\phi_1)&[\tau(\xi)](\phi_1')\endbmatrix} \mathcal A_{\,x}\\ &\longrightarrow \frac{\mathcal A_{\,x}}{([\tau(\xi)](\phi_1)\ ,\ [\tau(\xi)](\phi_1'))\mathcal A_{\,x}} \to 0.\end{align*}
\item\label{*.critical.b} 
If $R_0$ is a Cohen-Macaulay domain, then
 the localization  {\rm(\ref{*.claim.1})}$_x$ is exact.\end{enumerate}\end{lemma}

\begin{remark-no-advance} Once we prove Theorem~\ref{main-Theorem}, then a much stronger version of Lemma~\ref{*.critical}
is also true, see Proposition~\ref{improved}.
\end{remark-no-advance}
\begin{proof} (\ref{*.critical.a}) The element $x$ in $\mathcal R$ is a non-zero element of $\mathcal R_{\ (1,0)}$. The ideal $I\cdot \mathcal R$ of $\mathcal R$ is a prime ideal generated by elements of $\mathcal R_{\ (2,0)}$; hence $x$ is a non-zero-divisor   in $\mathcal A=\mathcal R/(I\cdot \mathcal R)$. Consider the  map
\begin{equation}\label{*.natural-map}\mathcal A_{\,x}\oplus \mathcal A_{\,x}\longrightarrow 
\mathcal N_{\ \,x},\end{equation}which sends $\bmatrix a_1&a_2\endbmatrix^{\rm T}$ to the class of $a_1 \phi_1+a_2\phi_1'$. This  map is onto because, if $\phi_1''\in \mathcal F$, then the equation
\begin{equation}\label{*.because}0=\xi(\phi_1\wedge \phi_1'\wedge \phi_1'')=x\cdot\phi_1''-\xi(\phi_1\wedge \phi_1'')\cdot \phi_1'+\xi(\phi_1'\wedge \phi_1'')\cdot \phi_1\end{equation}
holds in $\mathcal N$ (see Observation \ref{doo-8.2}.\ref{doo-8.2.b}); and therefore the class of  $\phi_1''$ in $\mathcal N_{\ \,x}$   is in the image of the map (\ref{*.natural-map}). Let $(0)$ be the prime ideal $(0)$ in the domain $\mathcal A$ and $L$ be the kernel of (\ref{*.natural-map}).  We know from Lemma~\ref{exact-seq} that $\mathcal N_{\ \,(0)}=\mathcal A_{\,(0)}\oplus \mathcal A_{\,(0)}$; hence $L_{(0)}=0$. On the other hand, $L$ is a submodule of a free $\mathcal A_{\, x}$-module and $\mathcal A_{\,x}$ is a domain; thus, $L=0$ and (\ref{*.natural-map}) is an isomorphism. 

Apply (\ref{*.because}), with $\tau$ in place of $\phi_1''$, to see that the composition 
$$\mathcal A_{\,x}\xrightarrow{x}\mathcal A_{\,x}\xrightarrow {\tau} \mathcal N_{\ \,x}$$ sends $1\in \mathcal A_{\,x}$ to
$$x\tau=[\tau(\xi)](\phi_1)\cdot \phi_1'-[\tau(\xi)](\phi_1')\cdot \phi_1$$
in $\mathcal N_{\ \,x}$; and therefore, the composition 
$$\mathcal A_{\,x}\xrightarrow{x}\mathcal A_{\,x}\xrightarrow {\tau} \mathcal N_{\ \,x}\xrightarrow {\text{(\ref{*.natural-map})}^{-1}}\mathcal A_{\,x}\oplus \mathcal A_{\,x}$$ sends $1\in \mathcal A_{\,x}$ to 
$$\bmatrix -[\tau(\xi)](\phi_1')\\\phantom{-}[\tau(\xi)](\phi_1)\endbmatrix\in \mathcal A_{\,x}\oplus \mathcal A_{\,x}.$$ It is clear that the composition
$$\mathcal A_{\,x}\oplus \mathcal A_{\,x}\xrightarrow{\text{(\ref{*.natural-map})}} \mathcal N_{\ \,x}\xrightarrow{\tau(\xi)} \mathcal A_{\,x}$$ sends 
$$\bmatrix 1\\0\endbmatrix \mapsto [\tau(\xi)](\phi_1)\quad \text{and} \quad \bmatrix 0\\1\endbmatrix \mapsto [\tau(\xi)](\phi_1').$$ This completes the proof of (\ref{*.critical.a}).

\medskip\noindent(\ref{*.critical.b}) We know from (\ref{*.critical.a}) that the ideal $J\mathcal A_{\,x}$ is generated by $[\tau(\xi)](\phi_1)$ and $[\tau(\xi)](\phi_1')$ and we know from Lemma~\ref{*.enough} that $2\le \htt(J\mathcal A)$. The ring $\mathcal A$ is Cohen-Macaulay; so, 
$$2\le \htt(J\mathcal A)=\grade J\mathcal A\le \grade J\mathcal A_{\,x}.$$It follows that   {\rm(\ref{*.claim.1})}$_x$, which, according to (\ref{*.critical.a}), is isomorphic to the augmented Koszul complex on the generating set  $\{[\tau(\xi)](\phi_1)\ ,\ [\tau(\xi)](\phi_1')\}$ of $J\mathcal A_{\,x}$, is exact. 
\end{proof}

\begin{chunk}\label{*.proof-of-claim}{\bf The proof of Proposition~\ref{*.claim}.} In light of Remark~\ref{R4}.\ref{funny}, we may assume that ${4\le \goth f}$. 
We know from  Observation \ref{*.10.9} that (\ref{*.claim.1}) is a complex of $\mathcal A$-modules which is exact everywhere except possibly at $\mathcal N$. Let $\HH$ be the homology of (\ref{*.claim.1}) at $\mathcal N$. We argue by contradiction. Assume that $\HH\neq 0$. Let $P$ be an associated prime of $\HH$. 
Lemma~\ref{*.critical} shows 
that $\HH_x=0$ for every $x$ in $\mathcal R$ of the form 
\begin{equation}\label{form}\text{$x=\xi(\phi_1\wedge \phi_1')$ where $\phi_1$ and $\phi_1'$ are in $\mathcal F^*$ with $\phi_1\wedge \phi_1'$ part of a basis for $\textstyle\bigwedge^2\mathcal F^*$.}\end{equation} The fact that $\HH_x=0$ and $\HH_P\neq 0$ forces $x$ to be an element of $P$. The $R_0$-module  $\mathcal R_{(1,0)}$ is generated by elements $x$ of the form (\ref{form}); therefore, $\mathcal R_{\ (1,0)}\subseteq P$.

Consider the complex (\ref{*.claim.1}). Let  $B$ be the image of $\tau:\mathcal A\to \mathcal N$ and   $Z$ be the kernel of $\tau(\xi):\mathcal N\to \mathcal A$. Combine the exact sequences 
\begin{align*}&0\to \mathcal A\to B\to 0&&\text{from Observation~\ref{*.10.9},}\quad\text{and}\\
&0\to B\to Z\to \HH\to 0\end{align*}
in order to obtain 
the exact sequence \begin{equation}\label{*.ses}0\to \mathcal A\to Z\to \HH\to 0.\end{equation} The $\mathcal R$-modules $\mathcal A$  and $\mathcal N$ are both perfect 
and their annihilators have grade $\binom{\goth f-2}2$; see Corollary~\ref{carry forward}.
The ring $\mathcal R$ is Cohen-Macaulay; so, 
$\mathcal A_{\,P}$ and $\mathcal N_{\ \,P}$ are both Cohen-Macaulay $\mathcal R_{\ P}$-modules with 
$$\depth \mathcal N_{\ \,P}=\dim \mathcal N_{\ \,P}=\dim \mathcal A_{\,P}=\depth \mathcal A_{\,P};$$ and this common number is equal to $\dim \mathcal R_{\ P}-\binom{\goth f-2}2$. Furthermore,  the ideal $(\mathcal R_{\ 1,0})$ of $\mathcal R$, which is prime of height $\binom{\goth f}2$, is contained in $P$. 
It follows that
$$\textstyle 2\le \binom{\goth f}2-\binom{\goth f-2}2\le \dim \mathcal A_{\,P}.$$(The left most inequality holds because $3\le \goth f$.)
 The module $Z_P$ is a non-zero submodule of $\mathcal N_{\ \,P}$; so $1\le \depth Z_P$. We have chosen $P$ with $\HH_P\neq 0$ and $\depth \HH_P=0$. The usual argument about the growth of depth in a short exact sequence shows that the exact sequence 
$$0\to \mathcal A_{\,P}\to Z_P\to \HH_P\to 0,$$ which is obtained by localizing the short exact sequence (\ref{*.ses}) at $P$,
  is impossible; see, for example, \cite[Prop.~1.2.9]{BH}. This contradiction establishes the result. \qed
\end{chunk} 

\begin{theorem}\label{main-Theorem}Adopt the language of {\rm\ref{data2}} and {\rm \ref{Not2}}. If $4\le \goth f$ and $R_0$ is an arbitrary commutative Noetherian ring, then $J$ is a perfect Gorenstein ideal of $\mathcal R$ of grade $\binom{\goth f-2}2+2$. In particular, if $R_0$ is a Gorenstein ring, then $\mathcal R/ J$ is a Gorenstein ring.  \end{theorem}

\begin{proof} We employ the theory of generic perfection as described at the beginning of \ref{prove-dream}. It suffices to prove the result when $R_0$ is equal to the ring of integers and when $R_0$ is a field. In particular, we may assume that $R_0$ is a Cohen-Macaulay domain. Proposition~\ref{*.claim} guarantees that there exists an exact sequence of $\mathcal R$-modules 
$$0\to \mathcal A\to \mathcal N\to \mathcal A\to \mathcal R/J\to 0$$
and Corollary~\ref{carry forward} ensures that $\mathcal A$ and $\mathcal N$ have free resolutions of length $\binom{\goth f-2}2$; furthermore, the back Betti number in the resolution of $\mathcal A$ is one. Resolve $\mathcal A$ and $\mathcal N$ and form the iterated mapping cone in order to find a free resolution of $\mathcal R/J$ of length $\binom{\goth f-2}2+2$. The back Betti number in the resolution of $\mathcal R/J$ is one. We see that
 $$\textstyle \binom{\goth f-2}2+2\le \grade J\le \pd_{\mathcal R}\mathcal R/J\le \binom{\goth f-2}2+2.$$(The first inequality is Lemma~\ref{*.enough} and the second inequality is (\ref{auto}).) Thus, equality holds throughout and the proof is complete.
\end{proof}

\section{Consequences of the main result.}\label{consequences}

In this section, especially in Corollary~\ref{corollary}, we prove some consequences of the fact that $J$ is a perfect ideal in $\mathcal R$.
We begin by identifying some relations on the generators of $J$. These relations are used in the proof of Corollary~\ref{corollary}.\ref{cor-b} that $(\mathcal R/J)_{x_{i,j}}$ is a polynomial ring over $R_0[x_{i,j}\ ,\ x_{i,j}^{-1}]$.

\begin{definition}\label{5.1}Adopt the language of {\rm\ref{data2}} and {\rm\ref{Not2}}. Define the maps and modules 
$$\mathbb E_2\xrightarrow{D_2}\mathbb E_1\xrightarrow{D_1}\mathbb E_0$$
by \begingroup\allowdisplaybreaks\begin{align*}&\mathbb E_2= \begin{matrix}\bigwedge^3\mathcal F^*
\\\oplus\\ \ker\left(\mathcal F^*\otimes \bigwedge^5\mathcal F^*\xrightarrow{\operatorname{multiplication}}\bigwedge^6\mathcal F^*\right),
\\\oplus\\\bigwedge^3\mathcal F^*\otimes \bigwedge^3 \mathcal F^*\end{matrix}\quad\mathbb E_1= \begin{matrix}\mathcal F^*\\\oplus\\\bigwedge^4\mathcal F^*\end{matrix},\quad\mathbb E_0=\mathcal R,\\
&D_2\left(\bmatrix \phi_3\\0\\\phi_3'\otimes \phi_3'' \endbmatrix\right)=\bmatrix \xi(\phi_3)\hfill\\\tau\wedge \phi_3+\xi(\phi_3')\wedge \phi_3''-\phi_3'\wedge \xi(\phi_3'')\endbmatrix,\\&D_1\left(\bmatrix \phi_1\\\phi_4 \endbmatrix\right)=[\tau(\xi)](\phi_1)+\xi^{(2)}(\phi_4),\end{align*}\endgroup and the middle component of $D_2$ is induced by the map $F^*\otimes \bigwedge^5\mathcal F^*\to \bigwedge^4\mathcal F^*$ which sends $\phi_1\otimes \phi_5$ to $[\phi_1(\xi)](\phi_5)$. 
\end{definition}

\begin{observation}\label{D}The maps and modules of Definition~{\rm\ref{5.1}} form a complex and the image of $D_1$ is the ideal $J$ of {\rm\ref{Not2}}.\end{observation}
\begin{proof}We verify that $D_1\circ D_2=0$. We use (\ref{Gamma}), (\ref{compat}), Observation~\ref{doo-8.2}.\ref{doo-8.2.a}, and the module action of $\bigwedge^{\bullet}\mathcal F$ and $\bigwedge^{\bullet}\mathcal F^*$ on one another to compute 
\begingroup\allowdisplaybreaks\begin{align*}(D_1\circ D_2)(\phi_3)&=D_1\left(\bmatrix \xi(\phi_3)\\\tau\wedge \phi_3\endbmatrix\right)=[\tau(\xi)](\xi(\phi_3))+\xi^{(2)}(\tau\wedge \phi_3)\\&=
[\tau(\xi^{(2)})](\phi_3)-(\phi_3\wedge \tau)(\xi^{(2)})=0, \\
(D_1\circ D_2)(\sum_i\phi_{1,i}\otimes \phi_{5,i})&=\sum_i D_1([\phi_{1,i}(\xi)](\phi_{5,i}))=\sum_i \xi^{(2)}([\phi_{1,i}(\xi)](\phi_{5,i}))\\
&=\sum_i [\phi_{1,i}(\xi^{(3)})](\phi_{5,i})=(\sum_i\phi_{5,i} \wedge \phi_{1,i})(\xi^{(3)})=0, \quad\text{and}\\
(D_1\circ D_2)(\phi_3'\otimes \phi_3'')&=D_1\big(\xi(\phi_3')\wedge \phi_3''-\phi_3'\wedge \xi(\phi_3'')\big)=\xi^{(2)}\big(\xi(\phi_3')\wedge \phi_3''-\phi_3'\wedge \xi(\phi_3'')\big)\\&=\xi^{(2)}\big(\xi(\phi_3')\wedge \phi_3''\big)-\xi^{(2)}\big(\phi_3'\wedge \xi(\phi_3'')\big).
\\\intertext{Furthermore, we compute}
\xi^{(2)}\big(\xi(\phi_3')\wedge \phi_3''\big)&=-[\phi_3''\wedge \xi(\phi_3')](\xi^{(2)})=-\phi_3''\big([\xi(\phi_3')](\xi^{(2)})\big)=-\phi_3''\big([\xi(\phi_3')](\xi)\wedge \xi\big)\\&=-\phi_3''\big(
\phi_3'(\xi^{(2)})\wedge \xi\big)=-[\phi_3'(\xi^{(2)})](\xi(\phi_3''))=
-[\xi(\phi_3'')\wedge \phi_3'](\xi^{(2)})\\
&=[\phi_3'\wedge \xi(\phi_3'')](\xi^{(2)})=\xi^{(2)}[\phi_3'\wedge \xi(\phi_3'')];
\end{align*}\endgroup
and therefore, $(D_1\circ D_2)(\phi_3'\otimes \phi_3'')=0$.
\end{proof} 

\begin{corollary}\label{corollary}Adopt the language of {\rm\ref{data2}},  {\rm \ref{Not2}}, and {\rm \ref{R2}}. Assume that  $4\le \goth f$ and $R_0$ is an arbitrary commutative Noetherian ring. The following statements hold.
\begin{enumerate}[\rm(a)]
\item\label{cor-a} The elements $x_{1,2}\ ,\ x_{1,3}$ form a regular sequence 
on $\mathcal R/J$.
\item\label{cor-b} For each pair $i,j$ with $1\le i<j\le \goth f$, the localization of the ring $\mathcal R/J$ at the element $x_{i,j}$ is isomorphic to a polynomial ring over $R_0[x_{i,j}\ ,\ x_{i,j}^{-1}]$.
\item\label{cor-c} The ring $\mathcal R/J$ is a domain if and only if $R_0$ is a domain. 
\item\label{cor-c.5} If $R_0$ is a domain, then $(x_{1,2})\mathcal R/J$ is a prime ideal in $\mathcal R/J$.
\item\label{cor-d} The ring $\mathcal R/J$ is normal if and only if $R_0$ is normal. 
\item \label{cor-e} If $R_0$ is a normal domain, then the divisor class group of $R_0$ is isomorphic to the divisor class group of $\mathcal R/J$. In particular, $R_0$ is a unique factorization domain if and only if $\mathcal R/J$ is a unique factorization domain. \end{enumerate}
\end{corollary}
\begin{proof}(\ref{cor-a}) 
We employ the method of proof that is described in Corollary~\ref{KL-consq}. It suffices to show that 
\begin{align*}
&\textstyle \binom{\goth f-2}2+3\le \grade PR_P&&\text{for all $P\in \Spec \mathcal R$ with $J+(x_{1,2})\subseteq P$}&&\text{situation 1, and}\\
&\textstyle\binom{\goth f-2}2+4\le \grade PR_P&&\text{for all $P\in \Spec \mathcal R$ with $J+(x_{1,2}\ ,\ x_{1,3})$}\subseteq P&&\text{situation 2}\\
\end{align*}
Fix a prime $P$ from situation 1 or situation 2.
There are two cases. Assume first that $t_{\goth f}\notin P$. The ring $\mathcal R_{\ P}$ is a localization of $\mathcal R_{\ t_f}$ and $\mathcal R_{\ t_f}$ is equal to the polynomial ring $$(R_0[t_1\ ,\ \dots\ ,\ t_\goth f\ ,\ t_\goth f^{-1},\{x_{i,j}\mid 1\le i<j\le \goth f-1\}])[(\mathbf{tX})_1\ ,\ \dots\ ,\ (\mathbf{tX})_{\goth f-1}].$$
Let $\mathbf X'$ represent $\mathbf X$ with row and column $\goth f$ deleted. Apply Corollary ~\ref{KL-consq}. In situation 1,
the ideal $(x_{1,2},J)\mathcal R_{\ t_\goth f}$ contains the $\grade \binom{\goth f-3}2+1$ ideal $(x_{1,2}\ ,\ \Pf_{4}(\mathbf X')$ of  \begin{equation}\label{ring}R_0[t_1\ ,\ \dots\ ,\ t_\goth f\ ,\ t_\goth f^{-1}\ ,\ \{x_{i,j}\mid 1\le i<j\le \goth f-1\}]\end{equation} as well as the $\goth f-1$ indeterminates $(\mathbf{tX})_1,\dots,(\mathbf{tX})_{\goth f-1}$. Thus, 
$$\textstyle \binom{\goth f-2}2+3=\Big(\binom{\goth f-3}2+1\Big)+(\goth f-1)\le \grade(x_{1,2}\ ,\ J)\mathcal R_{\ t_\goth f}.$$Similarly, in situation 2,  Corollary ~\ref{KL-consq} guarantees that 
$$\textstyle \binom{\goth f-3}2+2\le \grade (x_{1,2}\ ,\ x_{1,3}\ ,\ \Pf_{4}(\mathbf X'))\cdot \text{(\ref{ring})};$$so, $\binom{\goth f-2}2+4\le \grade (x_{1,2}\ ,\ x_{1,3}\ ,\ J)\mathcal R_{\ t_{\goth f}}$. 
 The same argument works if $t_{\goth f-1}\notin P$. The second case is $t_{\goth f}$ and $t_{\goth f-1}$ are both in $P$.  In this case, Corollary~\ref{KL-consq} yields
\begin{align*} &(t_{\goth f}\ ,\ t_{\goth f-1})+\Pf_4(\mathbf X)+(x_{1,2})\subseteq P&&\textstyle\text{and  $\binom{\goth f-2}2+3\le \grade P$}&&\text{in situation 1, and}\\
&(t_{\goth f}\ ,\ t_{\goth f-1})+\Pf_4(\mathbf X)+(x_{1,2}\ ,\ x_{1,3})\subseteq P&&\textstyle\text{and  $\binom{\goth f-2}2+4\le \grade P$}&&\text{in situation 2}
.\end{align*}  

\medskip\noindent(\ref{cor-b}) It is notationally convenient 
to  
prove the result for $(i,j)=(1,2)$. 
Let $S_1$ and $S_2$ be the following subsets of $\mathcal R$:
\begin{align}S_1=&\{x_{i,j}|1\le i\le 2\ ,\ 3\le j\le \goth f\}\cup \{t_j\mid 3\le j\le \goth f\}\quad\text{and} \label{Ssub1}\\S_2=&\textstyle\{x_{1,2}x_{i,j}-x_{1,i}x_{2,j}+x_{1,j}x_{2,i}\mid 3\le i<j\le \goth f\}\notag\\&\cup \Big\{x_{1,2}t_2+\sum\limits_{j=3}^{\goth f} x_{1,j}t_j\ ,\ x_{1,2}t_1-\sum\limits_{j=3}^{\goth f} x_{2,j}t_j\Big\}.\notag\end{align}

\noindent Notice that \begin{enumerate}[\rm(A)]
\item\label{cor-A} $S_1\cup S_2$ is a set of indeterminates over the ring $R_0[x_{1,2}\ ,\ x_{1,2}^{-1}]$,
\item\label{cor-B} $\big(R_0[x_{1,2}\ ,\ x_{1,2}^{-1}]\big)[S_1\cup S_2]=\mathcal R[x_{1,2}^{-1}]$, and
\item\label{cor-C} $J\mathcal R[x_{1,2}^{-1}]=(S_2)\mathcal R[x_{1,2}^{-1}]$. \end{enumerate}
Assertions (\ref{cor-A}) and (\ref{cor-B}) are obvious. 
Once  (\ref{cor-C}) is  established, we will have shown that
\begin{equation}\label{VARS} \begin{array}{l}\text{$(\mathcal R/J)_{x_{1,2}}$ is the polynomial ring }R_0[x_{1,2}\ ,\ x_{1,2}^{-1}][S_1]\text{ over $R_0[x_{1,2}\ ,\ x_{1,2}^{-1}]$}\vspace{5pt}\\\text{for $S_1$ given in (\ref{Ssub1})}.\end{array}\end{equation} 
We now prove (\ref{cor-C}). Observe first that $S_2\subset J$. Indeed, in the language of Observation~\ref{D} and Remark~\ref{R2}, the ideal $S_2\mathcal R$ is the image, under $D_1$, of the submodule $$\textstyle W=\mathcal R e_1^*\oplus \mathcal R e_2^*\oplus
\mathcal R (e_1^*\wedge e_2^*)\wedge \bigwedge^2 \mathcal F^*$$ of $\mathbb E_1$. We show that 
\begin{equation}\label{near}
\begin{array}{rcl}
x_{1,2} \mathcal F^* &\subseteq& W+\im D_2\\
x_{1,2} \mathcal R(e_1^*\ ,\ e_2^*)\wedge \bigwedge^3\mathcal F^* &\subseteq& W+\im D_2,\quad \text{and}\\
x_{1,2}\bigwedge^4\mathcal F^*&\subseteq& W+\mathcal R(e_1^*\ ,\ e_2^*)\wedge \bigwedge^3\mathcal F^*+\im D_2. \end{array}\end{equation}
Once (\ref{near}) is established, then iteration of (\ref{near}) gives $x_{1,2}^2\mathbb E_1\subseteq W+\im D_2$; hence, $x_{1,2}^2J$ is contained in $S_2 \mathcal R$ and (\ref{cor-C}) holds. 

If $\phi_1\in \mathcal F^*$, then use Observation~\ref{doo-8.2}.\ref{doo-8.2.b} to see that
\begin{align*}x_{1,2}\phi_1&=\xi(e_2^*\wedge e_1^*)\cdot \phi_1=\xi(\phi_1\wedge e_2^*\wedge e_1^*)+\xi(\phi_1\wedge e_1^*)\cdot  e_2^*-\xi(\phi_1\wedge e_2^*)\cdot  e_1^*\\
&=D_2(\phi_1\wedge e_2^*\wedge e_1^*)-\tau\wedge \phi_1\wedge e_2^*\wedge e_1^*+\xi(\phi_1\wedge e_1^*)\cdot  e_2^*-\xi(\phi_1\wedge e_2^*)\cdot  e_1^*
\in W+\im D_2.\end{align*}If $\phi_3\in \bigwedge^3\mathcal F^*$, then
\begin{align*}&x_{1,2}e_1^*\wedge \phi_3=\xi(e_2^*\wedge e_1^*)\cdot e_1^*\wedge \phi_3
=[e_1^*(\xi)](e_2^*\wedge e_1^*\wedge \phi_3)+\text{an element of $W$}\\
=&D_2(e_1^*\otimes e_2^*\wedge e_1^*\wedge \phi_3)+\text{an element of $W$}\in W+\im D_2.
\end{align*}The calculation $x_{1,2}e_2^*\wedge \phi_3\in W+\im D_2$ is similar.

\noindent
If $\phi_1\in \mathcal F^*$ and $\phi_3\in \bigwedge^3\mathcal F^*$, then \begingroup\allowdisplaybreaks
\begin{align*}{}&x_{1,2}\phi_1\wedge\phi_3=\xi(e_2^*\wedge e_1^*)\cdot \phi_1\wedge\phi_3\\
{}=&\xi(\phi_1\wedge e_2^*\wedge e_1^*)\wedge \phi_3
+\xi(\phi_1\wedge e_1^*)\cdot e_2^*\wedge\phi_3
-\xi(\phi_1\wedge e_2^*)\cdot e_1^*\wedge\phi_3\\
{}=&D_2((\phi_1\wedge e_2^*\wedge e_1^*)\otimes  \phi_3)
+\phi_1\wedge e_2^*\wedge e_1^*\wedge  \xi(\phi_3)
+ \text{an element of $\mathcal R(e_1^*\ ,\ e_2^*)\wedge \textstyle\bigwedge^3\mathcal F^*$}\\
{}\in&W+\mathcal R(e_1^*\ ,\ e_2^*)\wedge \textstyle\bigwedge^3\mathcal F^*+\im D_2.\end{align*}\endgroup This completes the proof of (\ref{near}) and hence the proof of (\ref{cor-b}).

\medskip\noindent(\ref{cor-c}) Apply (\ref{cor-a}) and then (\ref{cor-b}) to see that $\mathcal R/J$ is a domain if and only if $(\mathcal R/J)_{x_{1,2}}$ is a domain if and only if $R_0$ is a domain.

\medskip\noindent(\ref{cor-c.5}) Suppose $\alpha$ and $\beta$ are elements of $\mathcal R/J$ with $\alpha\beta\in (x_{1,2})\cdot \mathcal R/J$. We know from (\ref{cor-b}) that $(x_{1,2})\cdot (\mathcal R/J)_{x_{1,3}}$ is a prime ideal; so, one of the elements $\alpha$ or $\beta$ (say, $\alpha$) is in $(x_{1,2})\cdot (\mathcal R/J)_{x_{1,3}}$. It follows that $x_{1,3}^s\alpha\in  (x_{1,2})\cdot \mathcal R/J$, for some $s$. Apply (\ref{cor-a}) to see that $\alpha$ is in $(x_{1,2})\cdot \mathcal R/J$.

\medskip\noindent(\ref{cor-d}) ($\Leftarrow$) We apply the Serre criteria for normality in order to prove that $\mathcal R/J$ is normal. It suffices to prove that $(\mathcal R/J)_P$ is normal for all primes $P$ with $\depth(\mathcal R/J)_P\le 1$. If $\depth (\mathcal R/J)_P\le 1$, then (\ref{cor-a}) guarantees that at least one of the elements  $x_{1,2}$ or $x_{1,3}$ is not in $P$. Thus, we know from (\ref{cor-b}) that $(\mathcal R/J)_P$ is a localization of a polynomial ring over $R_0$; hence,  $(\mathcal R/J)_P$ is a normal domain. 

\medskip\noindent(\ref{cor-d}) ($\Rightarrow$) The hypothesis that $\mathcal R/J$ is normal guarantees that $\mathcal R/J$ is reduced; and therefore, $R_0$ is reduced. The localization $(\mathcal R/J)_{x_{1,2}}$ is also normal. Recall from 
(\ref{VARS}) that $(\mathcal R/J)_{x_{1,2}}$ is equal to $T[x_{1,2}^{-1}]$ where $T$ is the polynomial ring $R_0[x_{1,2}\ ,\ S_1]$ and $S_1$ is the list of indeterminates given in (\ref{Ssub1}). 
 Apply Lemma~\ref{BV}, with  $y=x_{1,2}$, to conclude  that $T$ is normal. Now a standard argument yields that $R_0$ is also normal.

\medskip\noindent(\ref{cor-e}) Avramov's proof \cite{A79} that $R/\Pf_{2t}(\mathbf X)$ is a unique factorization domain may be applied without change. In other words, there are isomorphisms of the following divisor class groups:
$$\xymatrix{\Cl(\mathcal R/J)\ar[r]^(.45){\alpha}& \Cl((\mathcal R/J)_{x_{1,2}})\ar[r]^(.48){\beta}& \Cl(R_0[S_1\ ,\ x_{1,2}^{-1}])
\ar@{<-}[r]^(.54){\gamma}
&
\Cl(R_0[S_1])\ar@{<-}[r]^(.54){\delta}&
\Cl(R_0).}$$ The element $x_{1,2}$ generates a prime ideal in  $\mathcal R/J$ by (\ref{cor-c.5}); so the isomorphism $\alpha$ is Nagata's Lemma \cite[Cor.~7.3]{F73}. We proved in (\ref{VARS}) that $(\mathcal R/J)_{x_{1,2}}$ is equal to the polynomial ring $R_0[S_1\ ,\ x_{1,2}^{-1}]$, where $S_1$ is the list of indeterminates given in (\ref{Ssub1}); so the isomorphism $\beta$ is the identity map. The isomorphism $\gamma$ is again Nagata's Lemma and the isomorphism $\delta$ is Gauss' Lemma \cite[Thm.~8.1]{F73}.
\end{proof}

We have used the following normality criterion which appears as \cite[Lemma~16.24]{BV}. The result follows quickly from Serre's normality criterion. \begin{lemma}\label{BV} Let $T$ be a Noetherian ring, and $y$ be a regular element of  $T$ such that  
$T/T y$ is reduced and $T[y^{-1}]$ is a normal ring. Then $T$ is a normal ring.
\end{lemma}

Now that we know that $J$ is a perfect ideal, we are able to improve some of the results that we used in order to prove that $J$ is perfect. Notice that there are no hypotheses on the ring $R_0$.

\begin{proposition}\label{improved} Adopt the language of {\rm\ref{data2}}, {\rm\ref{Not2}} and {\rm\ref{R2}}. Let $R_0$ be an arbitrary commutative Noetherian ring.  \begin{enumerate}[\rm(a)]

\item\label{improved.a}
The maps and modules of {\rm(\ref{*.claim.1})} form an exact sequence.

\item\label{improved.b} 
The maps and modules of {\rm(\ref{exact-seq.1})} form an exact sequence. 

\item\label{improved.X} The ideal $J\mathcal A_{\,x_{1,2}}$ is generated by the regular sequence $(\mathbf{tX})_1$, $(\mathbf{tX})_2$.

\item\label{improved.Y} The element $x_{1,2}$ of $R$ is regular on both $A$ and $N$ and $N_{x_{1,2}}\cong A_{x_{1,2}}\oplus A_{x_{1,2}}$.
\end{enumerate}
\end{proposition}

\begin{proof} (\ref{improved.a}) Let $\mathcal A$, $\mathcal N$, $\mathcal R$, and $J$ be the relevant modules built over $R_0$ and $\mathcal A_{\,\mathbb Z}$, $\mathcal N_{\ \,\mathbb Z}$, $\mathcal R_{\ \mathbb Z}$, and $J_{\mathbb Z}$ be the relevant modules built over ${\mathbb Z}$. We have shown in Proposition~\ref{*.claim} that
\begin{equation}\label{star}0\to \mathcal A_{\, \mathbb Z}\xrightarrow{\tau}
\mathcal N_{\ \,\mathbb Z}\xrightarrow{\tau(\xi)}\mathcal A_{\,\mathbb Z}\to \mathcal R_{\ \mathbb Z}/ J_{\mathbb Z}\to 0\end{equation} is an exact sequence. We know from Corollary~\ref{carry forward} and Theorem~\ref{main-Theorem} that $\mathcal A_{\,\mathbb Z}$, $\mathcal N_{\,\ \mathbb Z}$, and $\mathcal R_{\ \mathbb Z}/J_{\mathbb Z}$ are generically perfect $\mathbb Z[X]$-modules in the sense of \cite[Prop.~3.2 and Thm.~3.3]{BV}; and so, in particular, these modules  are flat $\mathbb Z$-modules. Apply $R_0\otimes_{\mathbb Z}-$ to the constituent short exact sequences of (\ref{star}) in order to learn that $\Tor_1^{\mathbb Z}(R_0,J_{\mathbb Z}\mathcal A_{\mathbb Z})=0$ and $R_0\otimes_{\mathbb Z}\text{(\ref{star})}$, which is isomorphic to (\ref{*.claim.1}),  is exact.

\medskip \noindent(\ref{improved.b}) The 
proof from (\ref{improved.a}) also works for (\ref{improved.b}) because the $\mathbb Z[X]$-modules $A$ and $A'$, built over $\mathbb Z$, are also generically perfect, see Lemma~\ref{KL}. 

\medskip \noindent(\ref{improved.X}) The proof of Corollary~\ref{corollary}.\ref{cor-b} shows that $\mathcal A_{\,x_{1,2}}$ is equal to the polynomial ring $$R_0[x_{1,2}\ ,\ x_{1,2}^{-1}][S_1\ ,\ (\mathbf{tX})_1\ ,\ (\mathbf{tX})_2],$$ where $S_1$ is the list of indeterminates given in (\ref{Ssub1}); furthermore, $J\mathcal A_{\,x_{1,2}}$ is generated by the two variables $(\mathbf{tX})_1$ and  $(\mathbf{tX})_2$.

\medskip \noindent(\ref{improved.Y}) We saw in Corollary~\ref{KL-consq} that $x_{1,2}$ is regular on $A$. Recall from Lemmas~\ref{main-Dream-Lemma} and \ref{KL}  that the ring $A$ and the $A$-module $N$ are perfect $R$-modules, and their annihilators  (as $R$-modules) have   the same grade. It follows that $\Ass N\subseteq \Ass A$ and that $x_{1,2}$ is also regular on $N$. The final assertion is obtained by localizing (\ref{exact-seq.1}), which is exact by (\ref{improved.b}), at $x_{1,2}$.
\end{proof}

\bigskip

\section{Remarks and questions.}\label{further}

\bigskip

The definition of $N$, as given in Notation~\ref{Not2}.\ref{Not2.a.iii}, is that $$\textstyle N=\frac RI\otimes_R\coker(d_1:\bigwedge^3F^*\to F^*).$$ However, if $2$ is a unit in $R_0$, then the next result shows that it  is not necessary to apply the functor $\frac RI\otimes_R-$.
\begin{observation}\label{two}Adopt the language of {\rm\ref{data2}.\ref{data2-one}}, {\rm\ref{Not2}.\ref{Not2.a}} and {\rm(\ref{pre-cplx})}. If $2$ is a unit in $R_0$, then 
  $$\textstyle \coker(d_1:\bigwedge^3F^*\to F^*)$$
is an $R/I$ module; so, in particular $N=\coker(d_1:\bigwedge^3F^*\to F^*)$.
\end{observation}
\begin{proof}If $\phi_4\in \bigwedge^4F^*$ and $\phi_1\in F^*$, then 
\begin{align*}\xi^{(2)}(\phi_4)\cdot \phi_1&= [\phi_1(\xi^{(2)})](\phi_4)+\xi^{(2)}(\phi_1\wedge \phi_4)&&\text{Proposition~\ref{A3}}\\
&= [\phi_1(\xi)\wedge \xi](\phi_4)+{\textstyle\frac 12}\xi(\xi(\phi_1\wedge \phi_4))&&\text{(\ref{Gamma})}\\
&=\xi\Big ([\phi_1(\xi)](\phi_4)+{\textstyle\frac 12}\xi(\phi_1\wedge \phi_4)\Big),\end{align*} which represents $0$ in $N$. \end{proof}
\begin{remarks}Adopt the language of {\rm\ref{data2}} and {\rm \ref{Not2}}. \begin{enumerate}[\rm(a)] \item The hypothesis ``$2$ is a unit in $R_0$'' is essential in Observation~\ref{two}. For example, if $R_0$ is the field $\mathbb Z/(2)$, then
$$\bmatrix 0&0&0&0&x_{1,2}x_{3,4}-x_{1,3}x_{2,4}+x_{1,4}x_{2,3}\endbmatrix^{\rm T}$$is zero in $N$, but is  not in image of $d_1$. So, in particular, if $R_0$ is a field, then the first Betti number of $N$, as a module over $R$, depends on the characteristic of $R_0$, even when $\goth f=5$. We recall that the first Betti number of $A$, as a module, over $R$ depends on the characteristic $R_0$, but not until $\goth f=8$; see, for example, \cite{K-I,K-II,H95}.

\item\label{R5.b} Assume $R_0$ is a field. Suppose that $\mathbb F:\ \dots \to F_i\to \dots$ and $\mathbb G:\ \dots \to G_i\to \dots$ are minimal homogeneous resolutions of $A$ and $N$ by free $R$-modules with $F_i=\bigoplus R(-j)^{\beta_{i,j}}$ and $G_i=\bigoplus R(-j)^{\gamma_{i,j}}$. Then the proof of Theorem~\ref{main-Theorem} shows that the minimal bi-homogeneous resolution of $\mathcal R/J$ by free $\mathcal R$-modules is $\mathbb L:\ \dots \to L_i\to \dots$, with
$$L_i= \bigoplus \mathcal R(-j-1,-2)^{\beta_{i-2,j}}\oplus \bigoplus \mathcal R(-j-1,-1)^{\gamma_{i-1,j}}\oplus 
\bigoplus \mathcal R(-j,0)^{\beta_{i,j}}.$$
Indeed, the  iterated mapping cone 
associated to 
$$\xymatrix{  \mathcal R(-1,-2)\otimes_R \mathbb F[-2]\ar[r]&\mathcal R(-1,-2)\otimes_R A \ar@{>->}[d]^{\tau}\\ 
\mathcal R(-1,-1)\otimes_R \mathbb G[-1]\ar[r]&\mathcal R(-1,-1)\otimes_R N \ar[d]^{\tau(\xi)} \\
\mathcal R\otimes_R \mathbb F\ar[r]&\mathcal R\otimes_R A\ar@{->>}[d]\\&\mathcal R/J}$$
is a bi-homogeneous resolution of $\mathcal R/J$ and consideration of the $t$-degree shows that this resolution is minimal.
\item Retain the language of (\ref{R5.b}). If the characteristic of $R_0$ is zero, then the resolution $\mathbb F$ is given in Theorem 6.4.1 and Exercises 31--33 on page 222 in \cite{W03}. Can the geometric method of \cite{W03} also be used to obtain  the minimal homogeneous equivariant resolution of $N$ by free $R$-modules?   
\item One consequence of (\ref{R5.b}) is that the minimal homogeneous resolution $\mathbb G$ of $N$ by free $R$-modules is self-dual. Is this fact obvious for
 some other reason?
\item Is the resolution $\mathbb X$ of $N$ by free $A$-modules from Observation~{\rm\ref{yet-to-come}.\ref{yet.e}} a linear complex?
\end{enumerate}\end{remarks}

\end{document}